\documentclass[12pt,a4paper,reqno]{amsart}
\usepackage{amsmath}
\usepackage{amsthm}
\usepackage{amsfonts}
\usepackage{amssymb}
\usepackage{color}
\usepackage{mathrsfs}
\usepackage{graphicx}

\numberwithin{equation}{section}

\theoremstyle{plain}
\newtheorem{thm}{Theorem}

\theoremstyle{plain}
\newtheorem{lem}{Lemma}[section]

\theoremstyle{plain}
\newtheorem*{theorema}{Theorem~A}

\theoremstyle{plain}
\newtheorem*{theoremb}{Theorem~B}

\theoremstyle{remark}

\theoremstyle{definition}
\newtheorem*{case1}{Case~1}

\theoremstyle{definition}
\newtheorem*{case2}{Case~2}

\theoremstyle{definition}
\newtheorem*{case2a}{Case~2A}

\theoremstyle{definition}
\newtheorem*{case2b}{Case~2B}

\def\bfs{\mathbf{s}}

\def\bfv{\mathbf{v}}
\def\bfw{\mathbf{w}}

\def\dd{\mathrm{d}}

\def\eps{\varepsilon}

\def\Rr{\mathbb{R}}

\def\Zz{\mathbb{Z}}

\def\AAA{\mathcal{A}}
\def\BBB{\mathcal{B}}

\def\EEE{\mathcal{E}}

\def\III{\mathcal{I}}
\def\JJJ{\mathcal{J}}

\def\LLL{\mathcal{L}}

\def\NNN{\mathcal{N}}

\def\PPP{\mathcal{P}}

\def\SSS{\mathcal{S}}
\def\TTT{\mathcal{T}}
\def\UUU{\mathcal{U}}
\def\VVV{\mathcal{V}}
\def\WWW{\mathcal{W}}
\def\XXX{\mathcal{X}}

\def\frakI{\mathfrak{I}}

\def\MMMM{\mathscr{M}}

\DeclareMathOperator{\length}{length}

\renewcommand{\le}{\leqslant}
\renewcommand{\ge}{\geqslant}

\title[Super-fast spreading]
{Super-fast spreading of\\
billiard orbits in rational polygons\\
and geodesics on translation surfaces}

\author[Beck]{J. Beck}

\address{Department of Mathematics, Hill Center for the Mathematical Sciences,
Rutgers University, Piscataway NJ 08854, USA}

\email{jbeck@math.rutgers.edu}

\author[Chen]{W.W.L. Chen}

\address{School of Mathematical and Physical Sciences, Faculty of Science and Engineering,
Macquarie University, Sydney NSW 2109, Australia}

\email{william.chen@mq.edu.au}



\begin{document}

\begin{abstract}
In this paper, we show that billiard orbits in rational polygons and geodesics on translation surfaces exhibit super-fast spreading,
an optimal time-quantitative majority property about the corresponding linear flow that implies uniformity in almost every direction.
\end{abstract}

\maketitle

\thispagestyle{empty}

%
%

\section{Introduction}\label{sec1}

A polygon is said to be rational if every angle is a rational multiple of~$\pi$.
It is well known, via the concept of unfolding introduced by K\"{o}nig and Sz\"{u}cs~\cite{KS13} in 1913 for the unit square
and extended by Fox and Kershner~\cite{FK36} in 1936, that billiard in a rational polygon
is equivalent to $1$-direction geodesic flow on a translation surface.
Here a translation surface is constructed from a finite collection of polygons on the plane,
together with appropriate pairings of sides of equal length and direction, \textit{i.e.} angle in the interval $[0,2\pi)$, identified via translation.
Geodesic flow on such a surface is therefore $1$-direction linear flow.

The first pioneering result on $1$-direction geodesic flow on translation surfaces is due to Katok and Zemlyakov~\cite{KZ75} in 1975,
and concerns density.

\begin{theorema}
Let $\PPP$ be a translation surface.
Then, apart from a countable set of directions, any $1$-direction geodesic is dense on $\PPP$
unless it hits a vertex of $\PPP$ and becomes undefined.
Furthermore, every exceptional direction is represented by a saddle connection of~$\PPP$.
\end{theorema}

Here, a saddle connection is a finite $1$-direction geodesic segment that joins two vertices, not necessarily distinct,
of the defining polygons of the translation surface.

The second pioneering result, due to Kerckhoff, Masur and Smillie~\cite{KMS86} in 1986, concerns the stronger property of uniformity.

\begin{theoremb}
Let $\PPP$ be a translation surface.
Then, for almost every direction, any $1$-direction geodesic is uniformly distributed on $\PPP$
unless it hits a vertex of $\PPP$ and becomes undefined.
\end{theoremb}

Unfortunately, these two extremely interesting and important results seem to have the same limitation,
in that they do not give any information on the time scale on the convergence to density and uniformity.
To address this limitation, Vorobets~\cite{vorobets97} has in 1997 established a time-quantitative version of Theorem~B, using new ideas.
Here, uniformity is tested with respect to \textit{large} sets, and the explicit error term given is rather weak.

In this paper, we use a method completely different from that of Vorobets.
Given an arbitrary translation surface, for the majority of directions, we can establish an optimal form of time-quantitative uniformity for the geodesics.
It is in fact an optimal form of \textit{local} uniformity, which nevertheless implies \textit{global} uniformity,
and we refer to it as \textit{super-fast spreading} of the geodesic flow.

To give an illustration of some optimal local properties concerning the distribution of geodesics, we turn to geodesic flow on the unit torus $[0,1)^2$.
This is an integrable dynamical system, and the unit torus is the simplest translation surface.
It has a basically complete time-quantitative theory, due to the classical work of Weyl, Hardy, Littlewood, Ostrowski, Khinchin and others.

Let $\LLL_\alpha(t)$, $t\ge0$, be a half-infinite geodesic with direction $(1,\alpha)$ on the unit torus $[0,1)^2$, with the usual arc-length parametrization.

(i)
The geodesic $\LLL_\alpha(t)$, $t\ge0$, is superdense on the unit torus $[0,1)^2$ if and only if $\alpha$ is a badly approximable number.
Here superdensity means that there exists a constant $C_1=C_1(\alpha)>0$ such that for every positive integer~$N$,
the finite geodesic segment $\LLL_\alpha(t)$, $0\le t\le C_1N$, gets $(1/N)$-close to every point of $[0,1)^2$.

(ii)
The geodesic $\LLL_\alpha(t)$, $t\ge0$, is super-micro-uniform on the unit torus $[0,1)^2$ if and only if $\alpha$ is a badly approximable number.
Here super-micro-uniformity means that for every $\eps>0$, there exists a constant $C_2=C_2(\alpha;\eps)>0$
such that for every positive integer $N$ and every aligned square $Q$ of side length~$1/N$,
\begin{displaymath}
\frac{(1-\eps)C_2}{N}<\vert\{0\le t\le C_2N:\LLL_\alpha(t)\in Q\}\vert<\frac{(1+\eps)C_2}{N}.
\end{displaymath}
Note that since every \textit{large} square can be decomposed into \textit{small} squares, it follows that super-micro-uniformity
implies uniformity in the sense of Weyl.

(iii)
Geodesic flow on the unit torus $[0,1)^2$ exhibits super-fast spreading.
Given any $\eps>0$, there exists an explicitly computable constant $C_3=C_3(\eps)$ such that for every positive integer~$N$,
there exists a set of directions $\Gamma(N;\eps)\subset[0,2\pi)$ with measure $\lambda(\Gamma(N;\eps))\ge(1-\eps)2\pi$
such that for any direction $\theta\in\Gamma(N;\eps)$ and any square $Q$ of side length $1/N$ on~$[0,1)^2$,
any geodesic segment $\LLL_\theta(t)$, $0\le t\le C_3N$, with direction $\theta$ and length~$C_3N$, satisfies
\begin{displaymath}
\frac{(1-\eps)C_3}{N}<\vert\{0\le t\le C_3N:\LLL_\theta(t)\in Q\}\vert<\frac{(1+\eps)C_3}{N}.
\end{displaymath}
Note first of all that super-fast spreading is a majority property about geodesic flow, and not about any geodesic with any given direction,
so it is quite different from superdensity and super-micro-uniformity.
On the other hand, if $\eps>0$, then intuitively $\eps$-uniformity tends to uniformity, and the exceptional set of angles tends to measure zero.
Indeed, super-fast spreading of geodesic flow on the unit torus $[0,1)^2$ implies uniformity in almost every direction.

We define super-fast spreading on an arbitrary translation surface $\PPP$ of area $1$ in an analogous way.
Given any $\eps>0$, there exists an explicitly computable constant $C_3=C_3(\PPP;\eps)$ such that for every positive integer~$N$,
there exists a set of directions $\Gamma(\PPP;N;\eps)\subset[0,2\pi)$ with measure $\lambda(\Gamma(\PPP;N;\eps))\ge(1-\eps)2\pi$
such that for any direction $\theta\in\Gamma(\PPP;N;\eps)$ and any square $Q$ of side length $1/N$ on~$\PPP$,
any geodesic segment $\LLL_\theta(t)$, $0\le t\le C_3N$, with direction $\theta$ and length~$C_3N$, satisfies
\begin{displaymath}
\frac{(1-\eps)C_3}{N}<\vert\{0\le t\le C_3N:\LLL_\theta(t)\in Q\}\vert<\frac{(1+\eps)C_3}{N}.
\end{displaymath}
Here we assume that the half-infinite geodesic $\LLL_\theta(t)$, $t\ge0$, does not hit a vertex of~$\PPP$.

The main result of this paper is the following.

\begin{thm}\label{thm1}
Let $\PPP$ be an arbitrary translation surface of area~$1$.
Then geodesic flow on $\PPP$ exhibits super-fast spreading.
\end{thm}

Super-fast spreading of billiard flow in a rational polygon can be defined in an analogous way.
It then follows as a straightforward consequence of Theorem~\ref{thm1} that billiard flow in an arbitrary rational polygon exhibits super-fast spreading.

%
%

\section{Geodesics on a translation surface}\label{sec2}

If we consider geodesic flow on the unit torus $[0,1)^2$, it is clear that rational slopes lead to periodic orbits.
Thus directions in $[0,2\pi)$ leading to rational slopes are the periodic directions.
Such directions can be characterized by line segments on the plane that join two lattice points, so the growth rate of the number of periodic orbits
is essentially a lattice point counting problem.

On a general translation surface, we have the additional problem of singularities, giving rise to saddle connections
that join vertices of the defining polygons of the translation surface.
As a polygon does not in general tile the plane, counting the number of periodic directions and saddle connections is much more complicated.
However, we have the following deep result of Masur~\cite{masur86,masur88,masur90}
on the quadratic growth rate of periodic directions and saddle connections.

For any translation surface $\PPP$ and any real number $T\ge0$, let $\NNN_1(\PPP;T)$ denote the set of directions $\phi$
such that there is a closed geodesic in direction $\phi$ on $\PPP$ with arc length not exceeding~$T$,
and let $\NNN_2(\PPP;T)$ denote the set of directions $\phi$
such that there is a saddle connection on $\PPP$ in direction $\phi$ and arc length not exceeding~$T$.

\begin{lem}\label{lem2.1}
For any translation surface $\PPP$ and any real number $T\ge0$, we have
\begin{equation}\label{eq2.1}
\vert\NNN_1(\PPP;T)\vert\le\vert\NNN_2(\PPP;T)\vert\le C^\star T^2,
\end{equation}
where the constant $C^\star=C^\star(\PPP)$ is independent of~$T$.
\end{lem}

As a consequence of lattice point counting in the special case of the unit torus, the quadratic upper bound in \eqref{eq2.1} is best possible.

We remark that Masur has also established a corresponding quadratic lower bound.
However, the proof of Masur is ineffective and establishes the existence of the constant factors only.
More recently, Vorobets~\cite{vorobets97} has obtained both bounds with effective constant factors.

For any translation surface $\PPP$ and any real number $T\ge0$, let $\NNN^*_1(\PPP;T)$ denote the set of directions $\phi$
such that there is a closed geodesic in direction $\phi$ on $\PPP$ with arc length greater than $T/2$ and not exceeding~$T$,
and let $\NNN^*_2(\PPP;T)$ denote the set of directions
$\phi$ such that there is a saddle connection on $\PPP$ in direction $\phi$ and arc length greater than $T/2$ and not exceeding~$T$.
Clearly
\begin{equation}\label{eq2.2}
\max\{\vert\NNN^*_1(\PPP;T)\vert,\vert\NNN^*_2(\PPP;T)\vert\}\le\vert\NNN_2(\PPP;T)\vert\le C^\star T^2.
\end{equation}

A translation surface $\PPP$ is a finite set
\begin{equation}\label{eq2.3}
\PPP=\{\PPP_1,\ldots,\PPP_\kappa\}
\end{equation}
of defining polygon faces on the plane, equipped with boundary identification of pairs of parallel edges of equal length,
leading to a compact, oriented and connected surface.

A half-infinite geodesic $\LLL_\theta(t)$, $t\ge0$, with direction $\theta$ on $\PPP$
can be viewed as an indexed collection of infinitely many parallel line segments inside the underlying bounded region of~$\PPP$.
Using repeated translations of the faces of~$\PPP$, we can visualize this half-infinite geodesic on $\PPP$
as a half-infinite straight line $L_\theta(t)$, $t\ge0$, on the plane with a rather straightforward procedure which we illustrate in Figure~1.

Suppose that the half-infinite geodesic $\LLL_\theta(t)$, $t\ge0$, on~$\PPP$ starts at a point $S$ inside a polygon face $\PPP_\ell$
of the collection \eqref{eq2.3} and has direction~$\theta$.
Let
\begin{displaymath}
P_0=P_0(\PPP_\ell;S)
\end{displaymath}
denote a copy of $\PPP_\ell$ on the plane, taking care to distinguish between the polygon face $\PPP_\ell$ of the translation surface $\PPP$
and the polygon $P_0(\PPP_\ell;S)$ on the plane.

\begin{displaymath}
\begin{array}{c}
\includegraphics{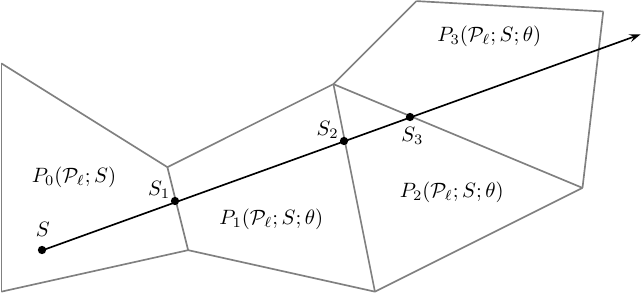}
\\
\mbox{Figure 1: converting a $1$-direction geodesic on a translation surface}
\\
\mbox{into a straight line on the plane}
\end{array}
\end{displaymath}

Let $t_1>0$ be minimal such that $\LLL_\theta(t_1)$ intersects a boundary edge of~$\PPP_\ell$, and let $S_1=\LLL_\theta(t_1)$.
This boundary edge of $\PPP_\ell$ is identified with a parallel edge of a polygon face $\PPP_{\ell_1}$ of the collection \eqref{eq2.3},
not necessarily distinct from~$\PPP_\ell$.
Let
\begin{displaymath}
P_1=P_1(\PPP_\ell;S;\theta)
\end{displaymath}
denote a copy of $\PPP_{\ell_1}$ on the plane, placed such that it is joined to $P_0=P_0(\PPP_\ell;S)$
along the two edges corresponding to the two identified edges of $\PPP_\ell$ and $\PPP_{\ell_1}$,
taking care to distinguish between the polygon face $\PPP_{\ell_1}$ of the translation surface $\PPP$
and the polygon $P_1=P_1(\PPP_\ell;S;\theta)$ on the plane.

Next, let $t_2>t_1$ be minimal such that $\LLL_\theta(t_2)$ intersects a boundary edge of~$\PPP_{\ell_1}$, and let $S_2=\LLL_\theta(t_2)$.
This boundary edge of $\PPP_{\ell_1}$ is identified with a parallel edge of a polygon face $\PPP_{\ell_2}$ of the collection \eqref{eq2.3},
not necessarily distinct from~$\PPP_{\ell_1}$.
Let
\begin{displaymath}
P_2=P_2(\PPP_\ell;S;\theta)
\end{displaymath}
denote a copy of $\PPP_{\ell_2}$ on the plane, placed such that it is joined to $P_1=P_1(\PPP_\ell;S;\theta)$ along the two edges
corresponding to the two identified edges of $\PPP_{\ell_1}$ and $\PPP_{\ell_2}$,
taking care to distinguish between the polygon face $\PPP_{\ell_2}$ of the translation surface
$\PPP$ and the polygon $P_2=P_2(\PPP_\ell;S;\theta)$ on the plane.

Next, let $t_3>t_2$ be minimal such that $\LLL_\theta(t_3)$ intersects a boundary edge of~$\PPP_{\ell_2}$, and let $S_3=\LLL_\theta(t_3)$.
This boundary edge of $\PPP_{\ell_2}$ is identified with a parallel edge of a polygon face $\PPP_{\ell_3}$ of the collection \eqref{eq2.3},
not necessarily distinct from~$\PPP_{\ell_2}$.
Let
\begin{displaymath}
P_3=P_3(\PPP_\ell;S;\theta)
\end{displaymath}
denote a copy of $\PPP_{\ell_3}$ on the plane, placed such that it is joined to $P_2=P_2(\PPP_\ell;S;\theta)$ along the two edges
corresponding to the two identified edges of $\PPP_{\ell_2}$ and $\PPP_{\ell_3}$,
taking care to distinguish between the polygon face $\PPP_{\ell_3}$ of the translation surface
$\PPP$ and the polygon $P_3=P_3(\PPP_\ell;S;\theta)$ on the plane.

And so on.
We thus have a sequence
\begin{displaymath}
P_0=P_0(\PPP_\ell;S),
\quad
P_i=P_i(\PPP_\ell;S;\theta),
\quad
i=1,2,3,
\end{displaymath}
of polygons on the plane, glued together along a half-infinite straight line $L_\theta(t)$, $t\ge0$, of direction $\theta$
and starting from the point $L_\theta(0)=S$.
We may call this the associated half-line of the geodesic $\LLL_\theta(t)$, $t\ge0$.
Furthermore, we call the line segments $S_{i-1}S_i$, $i=1,2,3,\ldots,$ the linear extensions of the half-infinite straight line $L_\theta(t)$, $t\ge0$,
with the convention that $S_0=S$.

Our plan is to define an interval exchange transformation to describe the effect of the geodesic flow on the edges
of the defining polygon faces in \eqref{eq2.3}.
However, before we do that, we need to first address a technical nuisance.
Suppose that $E_{i'}$ and $E_{i''}$ are two edges of the defining polygon faces of $\PPP$ in \eqref{eq2.3},
and that the image of a subinterval $J_{i'}\subset E_{i'}$ under geodesic flow in direction $\theta$ on the first edge $E_{i''}$
that the flow encounters is a subinterval $J_{i''}\subset E_{i''}$.
If the edges $E_{i'}$ and $E_{i''}$ happen to be parallel to each other, then the lengths of the subintervals $J_{i'}$ and $J_{i''}$ are equal.
However, this is not generally the case if $E_{i'}$ and $E_{i''}$ are not parallel to each other, as can be seen easily in Figure~2.
On the other hand, if we take a direction $H$ which is perpendicular to the flow, then the images of $J_{i'}$ and $J_{i''}$
under perpendicular projection on lines in this direction now have the same length.
In Figure~2, we have taken care to project the edges $E_{i'}$ and $E_{i''}$ on distinct lines $H_{i'}$ and~$H_{i''}$,
both of which are in the direction $H$ perpendicular to the flow.

\begin{displaymath}
\begin{array}{c}
\includegraphics{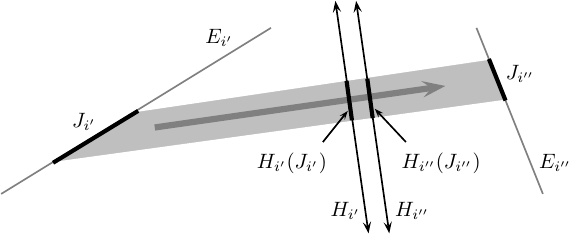}
\\
\mbox{Figure 2: a line perpendicular to the flow}
\end{array}
\end{displaymath}

Let $E_1,\ldots,E_b$ denote the edges of~$\PPP$, where $b=b(\PPP)$, and where each pair of identified edges is counted only once.
Let $H_1,\ldots,H_b$ denote distinct parallel lines in the direction $H$ perpendicular to the direction $\theta$ of the flow.

For each edge $E_i$, $i=1,\ldots,b$, let $H_i(E_i)$ denote the image of the perpendicular projection of $E_i$ on~$H_i$.
For convenience, we refer to this as the $H$-image of the edge~$E_i$, and may abuse notation by writing $H$ for~$H_i$.
Clearly the images $H_i(E_i)$, $i=1,\ldots,b$ are pairwise disjoint.

We shall avoid small neighborhoods of directions $\phi$ on $\PPP$ with saddle connections, and these include the directions of all the edges of~$\PPP$.
Thus we can ensure that
\begin{displaymath}
\vert H_i(E_i)\vert>0,
\quad
i=1,\ldots,b.
\end{displaymath}
The projection of the distinct edges of $\PPP$ to distinct lines in the direction $H$ now gives rise to mappings
\textcolor{white}{xxxxxxxxxxxxxxxxxxxxxxxxxxxxxx}
\begin{displaymath}
\psi_i:E_i\to H_i(E_i),
\quad
i=1,\ldots,b,
\end{displaymath}
which are bijective if we ignore the endpoints of the edges and their images.

Consider geodesic flow in a direction $\theta$ which is bounded away from the directions of the edges of~$\PPP$.
The union of the $H$-images is given by
\begin{displaymath}
H_0=H_0(\PPP;\theta)=H_1(E_1)\cup\ldots\cup H_b(E_b),
\end{displaymath}
and the sum of the lengths of the $H$-images is given by
\begin{displaymath}
c_1=c_1(\PPP;\theta)=\vert H_1(E_1)\vert+\ldots+\vert H_b(E_b)\vert.
\end{displaymath}
Thus we can assume that $H_0=[0,1)$, and identify the $H$-images
\begin{displaymath}
H_i(E_i),
\quad
i=1,\ldots,b,
\end{displaymath}
with $b$ pairwise disjoint subintervals
\begin{displaymath}
H'_i\subset[0,1),
\quad
\vert H'_i\vert=c_1^{-1}\vert H_i(E_i)\vert,
\quad
i=1,\ldots,b,
\end{displaymath}
if we ignore the vertices of $\PPP$ and their projection images.
Furthermore, if
\begin{displaymath}
\EEE=\EEE(\PPP)=E_1\cup\ldots\cup E_b
\end{displaymath}
denotes the union of the edges of~$\PPP$, then, somewhat abusing notation, we arrive eventually at a mapping
\textcolor{white}{xxxxxxxxxxxxxxxxxxxxxxxxxxxxxx}
\begin{displaymath}
\psi:\EEE\to[0,1),
\end{displaymath}
which is bijective if we ignore the vertices of $\PPP$ and their projection images.
It is clear that the restriction $\psi\vert_{E_i}=\psi_i$, $i=1,\ldots,b$.
Note that $\psi$ gives a one-to-one correspondence between the points on the edges of $\PPP$ and points on the unit interval $[0,1)$,
apart from the finitely many singularities arising from the vertices of~$\PPP$.

We comment here that the constant $c_1=c_1(\PPP;\theta)$ may differ from one direction $\theta$ to another.
However, there exist constants $c_2=c_2(\PPP)>0$ and $c_3=c_3(\PPP)>0$, depending only on~$\PPP$, such that
\begin{equation}\label{eq2.4}
c_2(\PPP)\le c_1(\PPP;\theta)\le c_3(\PPP),
\quad
\theta\in[0,2\pi).
\end{equation}
For instance, we can take $c_2=c_2(\PPP)$ to be the diameter of a circle lying within~$\PPP$,
and take $c_3=c_3(\PPP)$ to be the sum of the lengths of the edges of the defining polygons of~$\PPP$.

Let $c_0>1$ be an as yet unspecified constant.
For any real number $n$ and integer $m$ satisfying $1\le m\le n$, consider a \textit{bad} direction
\begin{displaymath}
\phi\in\NNN^*_1(\PPP;2^m)\cup\NNN^*_2(\PPP;2^m),
\end{displaymath}
and consider a short interval
\begin{equation}\label{eq2.5}
I_\phi(n;m;c_0)=\left[\phi-\frac{1}{c_02^{n+m}},\phi+\frac{1}{c_02^{n+m}}\right].
\end{equation}
Write
\textcolor{white}{xxxxxxxxxxxxxxxxxxxxxxxxxxxxxx}
\begin{equation}\label{eq2.6}
\Omega(n;c_0)=\bigcup_{1\le m\le n}\bigcup_{\phi\in\NNN^*_1(\PPP;2^m)\cup\NNN^*_2(\PPP;2^m)}I_\phi(n;m;c_0).
\end{equation}
Clearly it follows from \eqref{eq2.2} that
\begin{align}\label{eq2.7}
\lambda(\Omega(n;c_0))
&
\le\sum_{1\le m\le n}\sum_{\phi\in\NNN^*_1(\PPP;2^m)\cup\NNN^*_2(\PPP;2^m)}\vert I_\phi(n;m;c_0)\vert
\nonumber
\\
&
\le\sum_{1\le m\le n}2C^\star2^{2m}\frac{2}{c_02^{n+m}}
<\frac{8C^\star}{c_0}.
\end{align}

Suppose that $Q_0R_0$ is a subinterval of an edge $E_{i_0}$ of~$\PPP$.
The geodesic flow in the direction $\theta$ either moves $Q_0R_0$ to a subinterval $Q_1R_1$ of the first edge $E_{i_1}$ of $\PPP$
that the flow encounters, or there is splitting in the sense that the image is on two or more edges and includes a vertex of~$\PPP$.
Suppose that the former holds.
The geodesic flow in the direction $\theta$ then either moves $Q_1R_1$ to a subinterval $Q_2R_2$ of the first edge $E_{i_2}$ of $\PPP$
that the flow encounters, or there is splitting in the sense that the image is on two or more edges and includes a vertex of~$\PPP$.
Suppose again that the former holds.
The geodesic flow in the direction $\theta$ then either moves $Q_2R_2$ to a subinterval $Q_3R_3$ of the first edge $E_{i_3}$ of $\PPP$
that the flow encounters, or there is splitting in the sense that the image is on two or more edges and includes a vertex of~$\PPP$.
And so on.
Suppose now that the flow in the direction $\theta$ moves $Q_0R_0$ in this way free of splitting in $w$ consecutive forward extensions
to a subinterval $Q_wR_w$ of an edge $E_{i_w}$ of~$\PPP$, and that there is splitting in the next extension forward, hitting a vertex $V_1$ of~$\PPP$.
Figure~3 shows the analogous version of this on the plane.

\begin{displaymath}
\begin{array}{c}
\includegraphics{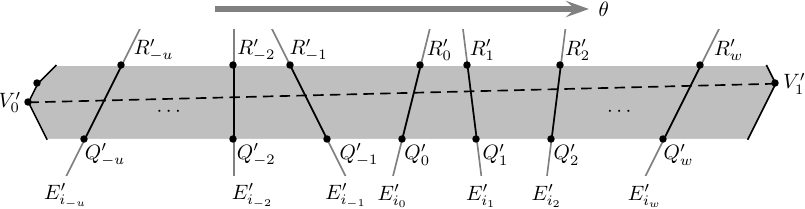}
\\
\mbox{Figure 3: transportation process: forward-backward extensions}
\end{array}
\end{displaymath}

On the other hand, the geodesic flow in the direction opposite to $\theta$ either moves $Q_0R_0$ to a subinterval $Q_{-1}R_{-1}$
of the first edge $E_{i_{-1}}$ of $\PPP$ that the flow encounters, or there is splitting in the sense that the image is on two or more edges
and includes a vertex of~$\PPP$.
Suppose that the former holds.
The geodesic flow in the direction opposite to $\theta$ then either moves $Q_{-1}R_{-1}$ to a subinterval $Q_{-2}R_{-2}$
of the first edge $E_{i_{-2}}$ of $\PPP$ that the flow encounters, or there is splitting in the sense
that the image is on two or more edges and includes a vertex of~$\PPP$.
And so on.
Suppose now that the flow in the direction opposite to $\theta$ moves $Q_0R_0$ in this way free of
splitting in $u$ consecutive backward extensions to a subinterval $Q_{-u}R_{-u}$ of an edge $E_{i_{-u}}$ of~$\PPP$,
and that there is splitting in the next extension backward, hitting a vertex $V_0$ of~$\PPP$.
Again Figure~3 shows the analogous version of this on the plane.

Clearly the finite geodesic segment $V_0V_1$ is a saddle connection of~$\PPP$.
It is easier to visualize this finite geodesic segment if we consider an infinite geodesic $\LLL_\theta(t)$, $t\in\Rr$,
on the surface $\PPP$ with $\LLL_\theta(0)=Q_0$.
This corresponds to a straight line $L_\theta(t)$, $t\in\Rr$, on the plane, and the saddle connection $V_0V_1$
corresponds to a straight line segment $V'_0V'_1$ on the plane, with length $\vert V'_0V'_1\vert$ satisfying
\begin{equation}\label{eq2.8}
c_4(u+w)>\vert V'_0V'_1\vert,
\end{equation}
where $c_4=c_4(\PPP)$ is an appropriate constant.
For instance, we can take $c_4(\PPP)$ to be twice the maximum of the diameters of the defining polygon faces of $\PPP$ in \eqref{eq2.3}.

\begin{lem}\label{lem2.2}
Suppose that $Q_0R_0$ is a subinterval of an edge $E_{i_0}$ of~$\PPP$, and consider the following transportation process.
The geodesic flow in the direction $\theta$ moves $Q_0R_0$ free of splitting in $w$ consecutive forward extensions
to a subinterval $Q_wR_w$ of an edge $E_{i_w}$ of~$\PPP$, and there is splitting in the next extension forward, hitting a vertex $V_1$ of~$\PPP$.
The geodesic flow in the direction opposite to $\theta$ moves $Q_0R_0$ free of splitting in $u$ consecutive backward extensions to a subinterval
$Q_{-u}R_{-u}$ of an edge $E_{i_{-u}}$ of~$\PPP$, and there is splitting in the next extension backward, hitting a vertex $V_0$ of~$\PPP$.

Suppose that $c_0\ge4$ is an as yet unspecified constant and $\theta\not\in\Omega(n_0;c_0)$, where the real number $n_0$ is fixed.
Suppose further that
\begin{equation}\label{eq2.9}
\frac{1}{c_0^22^{n_0+1}}\le\vert H(Q_0R_0)\vert<\frac{1}{c_0^22^{n_0}},
\end{equation}
where $\vert H(Q_0R_0)\vert=\vert H_{i_0}(Q_0R_0)\vert$ is the length of the $H$-image of the subinterval $Q_0R_0$.
Then
\textcolor{white}{xxxxxxxxxxxxxxxxxxxxxxxxxxxxxx}
\begin{displaymath}
u+w\ge c_52^{n_0},
\end{displaymath}
where $c_5=c_5(\PPP)>0$ is a constant.
\end{lem}

\begin{proof}
Let $m$ be the unique integer satisfying the inequalities
\begin{equation}\label{eq2.10}
2^{m-1}<\vert V'_0V'_1\vert\le2^m.
\end{equation}
It is clear from Figure~3 that the dashed line segment~$V'_0V'_1$, corresponding to the saddle connection $V_0V_1$ on~$\PPP$,
lies within a strip in direction $\theta$ on the plane with width $\vert H(Q_0R_0)\vert$,
and the direction $\phi=\phi(V'_0V'_1)$ of $V'_0V'_1$ satisfies
\begin{equation}\label{eq2.11}
\vert\sin(\phi-\theta)\vert\le\frac{\vert H(Q_0R_0)\vert}{\vert V'_0V'_1\vert},
\quad\mbox{so that}\quad
\vert\phi-\theta\vert\le\frac{2\vert H(Q_0R_0)\vert}{\vert V'_0V'_1\vert}.
\end{equation}
Combining \eqref{eq2.9}--\eqref{eq2.11}, we conclude that
\begin{equation}\label{eq2.12}
\vert\phi-\theta\vert<\frac{4}{c_0^22^{n_0+m}}.
\end{equation}
On the other hand, it follows from \eqref{eq2.10} that $\phi\in\NNN^*_2(\PPP;2^m)$.
Since $\theta\not\in\Omega(n_0;c_0)$, it follows from \eqref{eq2.5} and \eqref{eq2.6} that
\begin{equation}\label{eq2.13}
\vert\phi-\theta\vert>\frac{1}{c_02^{n_0+m}}.
\end{equation}
However, combining \eqref{eq2.12} and \eqref{eq2.13} leads to the inequality $c_0<4$, contradicting our assumption that $c_0\ge4$.
Note now that \eqref{eq2.5} and \eqref{eq2.6} are based on the assumption that $m\le n_0$, so we must have $m>n_0$.
Combining this fact with \eqref{eq2.8} and \eqref{eq2.10} now leads to the desired conclusion.
\end{proof}

The transportation process that we have described gives rise to subintervals $Q_jR_j$ on the edges $E_{i_j}$, where $-u\le j\le w$.
These intervals may overlap, meaning that there may exist $j_1,j_2$ satisfying $-u\le j_1<j_2\le w$
such that $E_{i_{j_1}}=E_{i_{j_2}}$ and the intervals $Q_{j_1}R_{j_1}$ and $Q_{j_2}R_{j_2}$ intersect.
The next lemma shows that this overlapping is limited in the quantitative sense that the union of the intervals $Q_jR_j$, $-u\le j\le w$,
is still \textit{large}.

\begin{lem}\label{lem2.3}
Suppose that $Q_0R_0$ is a subinterval of an edge $E_{i_0}$ of~$\PPP$, and consider the following transportation process.
The geodesic flow in the direction $\theta$ moves $Q_0R_0$ free of splitting in $w$ consecutive forward extensions
to a subinterval $Q_wR_w$ of an edge $E_{i_w}$ of~$\PPP$, and there is splitting in the next extension forward, hitting a vertex $V_1$ of~$\PPP$.
The geodesic flow in the direction opposite to $\theta$ moves $Q_0R_0$ free of splitting in $u$ consecutive backward extensions
to a subinterval $Q_{-u}R_{-u}$ of an edge $E_{i_{-u}}$ of~$\PPP$, and there is splitting in the next extension backward,
hitting a vertex $V_0$ of~$\PPP$.

Suppose that $c_0\ge4$ is an as yet unspecified constant and $\theta\not\in\Omega(n_0;c_0)$, where the real number $n_0$ is fixed.
Suppose further that the inequalities \eqref{eq2.9} hold.
Then there exists an integer constant $c_6=c_6(\PPP)$ such that $c_6<c_5$ and for any subset
\begin{displaymath}
\JJJ\subset\{j\in\Zz:-u\le j\le w\}
\end{displaymath}
of $c_62^{n_0}$ consecutive integers, the intervals $Q_jR_j$, $j\in\JJJ$, are pairwise disjoint.
\end{lem}

\begin{proof}
Suppose that there exist $j_1,j_2$ satisfying
\begin{displaymath}
-u\le j_1<j_2\le w,
\end{displaymath}
such that $E_{i_{j_1}}=E_{i_{j_2}}$ and the intervals $Q_{j_1}R_{j_1}$ and $Q_{j_2}R_{j_2}$ intersect.
Without loss of generality, assume that $Q_{j_2}\in Q_{j_1}R_{j_1}$, so that  $Q_{j_1}Q_{j_2}\subset Q_{j_1}R_{j_1}$.
Then it follows from \eqref{eq2.9} that
\textcolor{white}{xxxxxxxxxxxxxxxxxxxxxxxxxxxxxx}
\begin{equation}\label{eq2.14}
\vert H(Q_{j_1}Q_{j_2})\vert\le\vert H(Q_{j_1}R_{j_1})\vert=\vert H(Q_0R_0)\vert\le\frac{1}{c_0^22^{n_0}}.
\end{equation}
Furthermore, $Q_{j_2}\in Q_{j_1}R_{j_1}$ implies that there is another point $Q''_{j_2}\in Q'_{j_1}R'_{j_1}$
in the analogous version of the transportation process on the plane, as shown in Figure~4.

\begin{displaymath}
\begin{array}{c}
\includegraphics{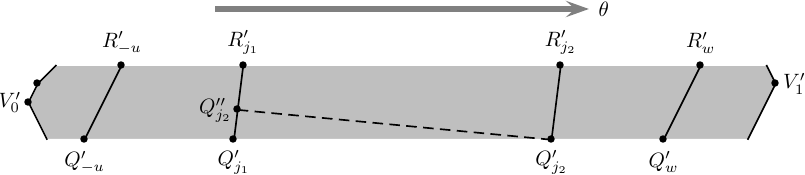}
\\
\mbox{Figure 4: overlapping edges}
\end{array}
\end{displaymath}

Let $m$ be the unique integer satisfying the inequalities
\begin{equation}\label{eq2.15}
2^{m-1}<\vert Q'_{j_2}Q''_{j_2}\vert\le2^m.
\end{equation}
It is clear from Figure~4 that the dashed line segment $Q'_{j_2}Q''_{j_2}$, corresponding to a periodic geodesic on $\PPP$
that goes through the point~$Q_{j_2}$, lies within a strip in direction $\theta$ on the plane with width $\vert H(Q_{j_1}Q_{j_2})\vert$,
and the direction $\phi=\phi(Q'_{j_2}Q''_{j_2})$ of $Q'_{j_2}Q''_{j_2}$ satisfies
\begin{equation}\label{eq2.16}
\vert\sin(\phi-\theta)\vert\le\frac{\vert H(Q_{j_1}Q_{j_2})\vert}{\vert Q'_{j_2}Q''_{j_2}\vert},
\quad\mbox{so that}\quad
\vert\phi-\theta\vert\le\frac{2\vert H(Q_{j_1}Q_{j_2})\vert}{\vert Q'_{j_2}Q''_{j_2}\vert}.
\end{equation}
Combining \eqref{eq2.14}--\eqref{eq2.16}, we conclude that
\begin{equation}\label{eq2.17}
\vert\phi-\theta\vert<\frac{4}{c_0^22^{n_0+m}}.
\end{equation}
On the other hand, it follows from \eqref{eq2.15} that $\phi\in\NNN^*_1(\PPP;2^m)$.
Since $\theta\not\in\Omega(n_0;c_0)$, it follows from \eqref{eq2.5} and \eqref{eq2.6} that
\begin{equation}\label{eq2.18}
\vert\phi-\theta\vert>\frac{1}{c_02^{n_0+m}}.
\end{equation}
However, combining \eqref{eq2.17} and \eqref{eq2.18} leads to the inequality $c_0<4$, contradicting our assumption that $c_0\ge4$.
Note now that \eqref{eq2.5} and \eqref{eq2.6} are based on the assumption that $m\le n_0$, so we must have $m>n_0$.
Combining this fact with \eqref{eq2.15} now leads to
\textcolor{white}{xxxxxxxxxxxxxxxxxxxxxxxxxxxxxx}
\begin{equation}\label{eq2.19}
\vert Q'_{j_2}Q''_{j_2}\vert>2^{n_0-1}.
\end{equation}
Clearly
\textcolor{white}{xxxxxxxxxxxxxxxxxxxxxxxxxxxxxx}
\begin{equation}\label{eq2.20}
\vert Q'_{j_2}Q''_{j_2}\vert\le c_7(\PPP)\vert j_1-j_2\vert,
\end{equation}
where $c_7(\PPP)$ is an upper bound on the diameters of the defining polygons of~$\PPP$.
Combining \eqref{eq2.19} and \eqref{eq2.20}, we conclude that $\vert j_1-j_2\vert>c_62^{n_0}$ for some constant $c_6=c_6(\PPP)$,
and leads to the desired conclusion.
\end{proof}

%
%

\section{Partitions, balance and anti-crowdedness}\label{sec3}

Let $M\ge1$ be an integer, and consider a set
\begin{displaymath}
\XXX=\{x_1,\ldots,x_M\}\subset[0,1).
\end{displaymath}
Suppose that $\delta\in(0,1)$ is arbitrarily small and fixed, and that $z\ge2$ is an integer.
For every integer $s\in\{1,\ldots,z\}$, consider the subinterval
\begin{displaymath}
I(s)=\left[\frac{s-1}{z},\frac{s}{z}\right)\subset[0,1).
\end{displaymath}
Note that $I(1)\cup\ldots\cup I(z)=[0,1)$ is a pairwise disjoint union.
We say that the set $\XXX$ is $(\delta;z)$-balanced if
\begin{displaymath}
-\frac{\delta M}{z}<\vert I(s)\cap\XXX\vert-\frac{M}{z}<\frac{\delta M}{z},
\quad
s=1,\ldots,z.
\end{displaymath}
Here the term $M/z=\vert[0,1)\cap\XXX\vert/z$ is the expectation of $\vert I(s)\cap\XXX\vert$.

Suppose that the integer $s_1\in\{1,\ldots,z\}$ is fixed.
For every integer $s\in\{1,\ldots,z\}$, consider the subinterval
\begin{displaymath}
I(s_1,s)=\frac{s_1-1}{z}+\left[\frac{s-1}{z^2},\frac{s}{z^2}\right)\subset I(s_1).
\end{displaymath}
Note that $I(s_1,1)\cup\ldots\cup I(s_1,z)=I(s_1)$ is a pairwise disjoint union.
We say that the set $\XXX$ is $(\delta;z)$-balanced relative to $I(s_1)$ if
\begin{displaymath}
-\frac{\delta M}{z^2}<\vert I(s_1,s)\cap\XXX\vert-\frac{\vert I(s_1)\cap\XXX\vert}{z}<\frac{\delta M}{z^2},
\quad
s=1,\ldots,z.
\end{displaymath}
Here the term $\vert I(s_1)\cap\XXX\vert/z$ is the expectation of $\vert I(s_1,s)\cap\XXX\vert$ relative to $\vert I(s_1)\cap\XXX\vert$.

Suppose next that $h\ge1$, and that the integers $s_1,\ldots,s_h\in\{1,\ldots,z\}$ are fixed.
Consider the interval
\begin{displaymath}
I(s_1,\ldots,s_h)=\frac{s_1-1}{z}+\ldots+\frac{s_{h-1}-1}{z^{h-1}}+\left[\frac{s_h-1}{z^h},\frac{s_h}{z^h}\right).
\end{displaymath}
For every integer $s\in\{1,\ldots,z\}$, consider the subinterval
\begin{displaymath}
I(s_1,\ldots,s_h,s)
=\frac{s_1-1}{z}+\ldots+\frac{s_h-1}{z^h}+\left[\frac{s-1}{z^{h+1}},\frac{s}{z^{h+1}}\right)
\subset I(s_1,\ldots,s_h).
\end{displaymath}
Note that $I(s_1,\ldots,s_h,1)\cup\ldots\cup I(s_1,\ldots,s_h,z)=I(s_1,\ldots,s_h)$ is a pairwise disjoint union.
We say that the set $\XXX$ is $(\delta;z)$-balanced relative to $I(s_1,\ldots,s_h)$ if
\begin{displaymath}
-\frac{\delta M}{z^{h+1}}<\vert I(s_1,\ldots,s_h,s)\cap\XXX\vert-\frac{\vert I(s_1,\ldots,s_h)\cap\XXX\vert}{z}<\frac{\delta M}{z^{h+1}},
\quad
s=1,\ldots,z.
\end{displaymath}
Here the term $\vert I(s_1,\ldots,s_h)\cap\XXX\vert/z$ is the expectation of $\vert I(s_1,\ldots,s_h,s)\cap\XXX\vert$ relative to
$\vert I(s_1,\ldots,s_h)\cap\XXX\vert$.

Suppose that the integer $p\ge1$ satisfies $z^p\le M_1$.
Write
\begin{align}
S_0(\XXX;p)
&
=\sum_{s_1=1}^z\ldots\sum_{s_p=1}^z
\left\vert\vert I(s_1,\ldots,s_p)\cap\XXX\vert-\frac{\vert[0,1)\cap\XXX\vert}{z^p}\right\vert^2,
\label{eq3.1}
\\
S_1(\XXX;p)
&
=\sum_{s_1=1}^z\ldots\sum_{s_p=1}^z
\left\vert\vert I(s_1,\ldots,s_p)\cap\XXX\vert-\frac{\vert I(s_1)\cap\XXX\vert}{z^{p-1}}\right\vert^2,
\label{eq3.2}
\\
S_2(\XXX;p)
&
=\sum_{s_1=1}^z\ldots\sum_{s_p=1}^z
\left\vert\vert I(s_1,\ldots,s_p)\cap\XXX\vert-\frac{\vert I(s_1,s_2)\cap\XXX\vert}{z^{p-2}}\right\vert^2,
\label{eq3.3}
\end{align}
and, in general, for any non-negative integer $h\le p$, write
\begin{equation}\label{eq3.4}
S_h(\XXX;p)
=\sum_{s_1=1}^z\ldots\sum_{s_p=1}^z
\left\vert\vert I(s_1,\ldots,s_p)\cap\XXX\vert-\frac{\vert I(s_1,\ldots,s_h)\cap\XXX\vert}{z^{p-h}}\right\vert^2.
\end{equation}

We state two very simple technical lemmas without proof.

\begin{lem}\label{lem3.1}
Suppose that $\VVV$ is a finite set, and that
\begin{displaymath}
A=\sum_{\bfv\in\VVV}a(\bfv),
\end{displaymath}
where $a$ is a real valued function on~$\VVV$.
Then
\begin{displaymath}
\sum_{\bfv\in\VVV}\left\vert a(\bfv)-\frac{A}{\vert\VVV\vert}\right\vert^2\le\sum_{\bfv\in\VVV}\vert a(\bfv)\vert^2.
\end{displaymath}
\end{lem}

\begin{lem}\label{lem3.2}
Suppose that $\VVV$ and $\WWW$ are finite sets, and that
\begin{displaymath}
B=\sum_{\bfv\in\VVV}B(\bfv)
\quad\mbox{and}\quad
B(\bfv)=\sum_{\bfw\in\WWW}b(\bfv,\bfw),
\end{displaymath}
where $b$ is a real valued function on $\VVV\times\WWW$.
Suppose further that
\begin{displaymath}
\TTT_0
=\sum_{\bfv\in\VVV}\sum_{\bfw\in\WWW}
\left\vert b(\bfv,\bfw)-\frac{B}{\vert\VVV\vert\vert\WWW\vert}\right\vert^2
\end{displaymath}
and
\textcolor{white}{xxxxxxxxxxxxxxxxxxxxxxxxxxxxxx}
\begin{displaymath}
\TTT_1
=\sum_{\bfv\in\VVV}\sum_{\bfw\in\WWW}
\left\vert b(\bfv,\bfw)-\frac{B(\bfv)}{\vert\WWW\vert}\right\vert^2.
\end{displaymath}
Then
\textcolor{white}{xxxxxxxxxxxxxxxxxxxxxxxxxxxxxx}
\begin{displaymath}
\TTT_0-\TTT_1
=\frac{1}{\vert\WWW\vert}\sum_{\bfv\in\VVV}
\left\vert B(\bfv)-\frac{B}{\vert\VVV\vert}\right\vert^2.
\end{displaymath}
\end{lem}

Applying Lemma~\ref{lem3.1}, we have
\begin{equation}\label{eq3.5}
S_0(\XXX;p)\le\sum_{s_1=1}^z\ldots\sum_{s_p=1}^z\vert I(s_1,\ldots,s_p)\cap\XXX\vert^2.
\end{equation}
Next, note that
\textcolor{white}{xxxxxxxxxxxxxxxxxxxxxxxxxxxxxx}
\begin{displaymath}
M=\vert[0,1)\cap\XXX\vert=\sum_{s_1=1}^z\vert I(s_1)\cap\XXX\vert.
\end{displaymath}
Applying Lemma~\ref{lem3.2} with
\begin{displaymath}
\bfv=s_1,
\quad
\bfw=(s_2,\ldots,s_p),
\quad
\VVV=\{1,\ldots,z\},
\quad
\WWW=\{1,\ldots,z\}^{p-1}
\end{displaymath}
on the sums $S_0(\XXX;p)$ and $S_1(\XXX;p)$, we deduce that
\begin{equation}\label{eq3.6}
S_0(\XXX;p)-S_1(\XXX;p)
=\frac{1}{z^{p-1}}\sum_{s_1=1}^z
\left\vert\vert I(s_1)\cap\XXX\vert-\frac{\vert[0,1)\cap\XXX\vert}{z}\right\vert^2
\ge0.
\end{equation}
Next, we can write
\begin{equation}\label{eq3.7}
S_1(\XXX;p)=\sum_{s_1=1}^zS_1(\XXX;p;s_1)
\quad\mbox{and}\quad
S_2(\XXX;p)=\sum_{s_1=1}^zS_2(\XXX;p;s_1),
\end{equation}
where
\textcolor{white}{xxxxxxxxxxxxxxxxxxxxxxxxxxxxxx}
\begin{displaymath}
S_1(\XXX;p;s_1)
=\sum_{s_2=1}^z\ldots\sum_{s_p=1}^z
\left\vert\vert I(s_1,\ldots,s_p)\cap\XXX\vert-\frac{\vert I(s_1)\cap\XXX\vert}{z^{p-1}}\right\vert^2
\end{displaymath}
and
\textcolor{white}{xxxxxxxxxxxxxxxxxxxxxxxxxxxxxx}
\begin{displaymath}
S_2(\XXX;p;s_1)
=\sum_{s_2=1}^z\ldots\sum_{s_p=1}^z
\left\vert\vert I(s_1,\ldots,s_p)\cap\XXX\vert-\frac{\vert I(s_1,s_2)\cap\XXX\vert}{z^{p-2}}\right\vert^2.
\end{displaymath}
Applying Lemma~\ref{lem3.2} with
\begin{displaymath}
\bfv=s_2,
\quad
\bfw=(s_3,\ldots,s_p),
\quad
\VVV=\{1,\ldots,z\},
\quad
\WWW=\{1,\ldots,z\}^{p-2}
\end{displaymath}
on the sums $S_1(\XXX;p;s_1)$ and $S_2(\XXX;p;s_1)$, we deduce that
\begin{displaymath}
S_1(\XXX;p;s_1)-S_2(\XXX;p;s_1)=\frac{1}{z^{p-2}}\sum_{s_2=1}^z
\left\vert\vert I(s_1,s_2)\cap\XXX\vert-\frac{\vert I(s_1)\cap\XXX\vert}{z}\right\vert^2.
\end{displaymath}
Combining this with \eqref{eq3.7}, we deduce that
\begin{equation}\label{eq3.8}
S_1(\XXX;p)-S_2(\XXX;p)=\frac{1}{z^{p-2}}\sum_{s_1=1}^z\sum_{s_2=1}^z
\left\vert\vert I(s_1,s_2)\cap\XXX\vert-\frac{\vert I(s_1)\cap\XXX\vert}{z}\right\vert^2
\ge0.
\end{equation}
In general, for any non-negative integer $h<p$, we can write
\begin{equation}\label{eq3.9}
S_h(\XXX;p)=\sum_{s_1=1}^z\ldots\sum_{s_h=1}^zS_h(\XXX;p;s_1,\ldots,s_h)
\end{equation}
and
\textcolor{white}{xxxxxxxxxxxxxxxxxxxxxxxxxxxxxx}
\begin{equation}\label{eq3.10}
S_{h+1}(\XXX;p)=\sum_{s_1=1}^z\ldots\sum_{s_h=1}^zS_{h+1}(\XXX;p;s_1,\ldots,s_h),
\end{equation}
where
\begin{displaymath}
S_h(\XXX;p;s_1,\ldots,s_h)
=\sum_{s_{h+1}=1}^z\ldots\sum_{s_p=1}^z
\left\vert\vert I(s_1,\ldots,s_p)\cap\XXX\vert-\frac{\vert I(s_1,\ldots,s_h)\cap\XXX\vert}{z^{p-h}}\right\vert^2
\end{displaymath}
and
\begin{displaymath}
S_{h+1}(\XXX;p;s_1,\ldots,s_h)
=\sum_{s_{h+1}=1}^z\ldots\sum_{s_p=1}^z
\left\vert\vert I(s_1,\ldots,s_p)\cap\XXX\vert-\frac{\vert I(s_1,\ldots,s_{h+1})\cap\XXX\vert}{z^{p-h-1}}\right\vert^2.
\end{displaymath}
Applying Lemma~\ref{lem3.2} with
\begin{displaymath}
\bfv=s_{h+1},
\quad
\bfw=(s_{h+2},\ldots,s_p),
\quad
\VVV=\{1,\ldots,z\},
\quad
\WWW=\{1,\ldots,z\}^{p-h-1}
\end{displaymath}
on the sums $S_h(\XXX;p;s_1,\ldots,s_h)$ and $S_{h+1}(\XXX;p;s_1,\ldots,s_h)$, we deduce that
\begin{align}
&
S_h(\XXX;p;s_1,\ldots,s_h)-S_{h+1}(\XXX;p;s_1,\ldots,s_h)
\nonumber
\\
&\quad
=\frac{1}{z^{p-h-1}}\sum_{s_{h+1}=1}^z
\left\vert\vert I(s_1,\ldots,s_{h+1})\cap\XXX\vert-\frac{\vert I(s_1,\ldots,s_h)\cap\XXX\vert}{z}\right\vert^2.
\nonumber
\end{align}
Combining this with \eqref{eq3.9} and \eqref{eq3.10}, we deduce that
\begin{align}\label{eq3.11}
&
S_h(\XXX;p)-S_{h+1}(\XXX;p)
\nonumber
\\
&\quad
=\frac{1}{z^{p-h-1}}\sum_{s_1=1}^z\ldots\sum_{s_{h+1}=1}^z
\left\vert\vert I(s_1,\ldots,s_{h+1})\cap\XXX\vert-\frac{\vert I(s_1,\ldots,s_h)\cap\XXX\vert}{z}\right\vert^2
\nonumber
\\
&\quad
\ge0.
\end{align}

Consider two real numbers $A\ge2$ and $M_1$ satisfying $1<M_1\le M$.
We say that the set $\XXX$ is \textit{$(A;M_1)$-anti-crowded} if for every subinterval $I\subset[0,1)$
satisfying the restriction $\vert I\vert\ge M_1^{-1}$, the number of elements of $\XXX$ in $I$ satisfies
\begin{displaymath}
\vert I\cap\XXX\vert\le AM\vert I\vert.
\end{displaymath}

Note that $\vert I(s_1,\ldots,s_p)\vert=z^{-p}\ge M_1^{-1}$.
If $\XXX$ is $(A;M_1)$-anti-crowded, then
\begin{displaymath}
\vert I(s_1,\ldots,s_p)\cap\XXX\vert\le\frac{AM}{z^p},
\end{displaymath}
and it follows from \eqref{eq3.5} that
\textcolor{white}{xxxxxxxxxxxxxxxxxxxxxxxxxxxxxx}
\begin{equation}\label{eq3.12}
S_0(\XXX;p)\le\frac{A^2M^2}{z^p}.
\end{equation}
Note that
\begin{equation}\label{eq3.13}
S_0(\XXX;p)
=S_0(\XXX;p)-S_p(\XXX;p)
=\sum_{h=0}^{p-1}(S_h(\XXX;p)-S_{h+1}(\XXX;p)),
\end{equation}
a sum of non-negative terms, in view of \eqref{eq3.6}, \eqref{eq3.8} and \eqref{eq3.11}.

For any set $\XXX\subset[0,1)$ of $M$ elements and any real number~$\tau$, let $\XXX+\tau$
denote the translated copy of $\XXX$ modulo the torus $[0,1)$.
Corresponding to \eqref{eq3.1}--\eqref{eq3.4}, for any non-negative integer $h\le p$, write
\begin{align}\label{eq3.14}
&
S_h(\XXX;\tau;p)
\nonumber
\\
&\quad
=\sum_{s_1=1}^z\ldots\sum_{s_p=1}^z
\left\vert\vert I(s_1,\ldots,s_p)\cap(\XXX+\tau)\vert-\frac{\vert I(s_1,\ldots,s_h)\cap(\XXX+\tau)\vert}{z^{p-h}}\right\vert^2
\nonumber
\\
&\quad
=\sum_{s_1=1}^z\ldots\sum_{s_p=1}^z
\left\vert\vert(I(s_1,\ldots,s_p)-\tau)\cap\XXX\vert-\frac{\vert(I(s_1,\ldots,s_h)-\tau)\cap\XXX\vert}{z^{p-h}}\right\vert^2.
\end{align}

Corresponding to \eqref{eq3.12} and \eqref{eq3.13}, we have the following integral version.

\begin{lem}\label{lem3.3}
Suppose that $\XXX\subset[0,1)$ is a set of $M$ elements.
Then
\begin{equation}\label{eq3.15}
\int_0^1S_0(\XXX;\tau;p)\,\dd\tau=z^h\int_0^{z^{-h}}S_0(\XXX;\tau;p)\,\dd\tau,
\quad
h=0,1,\ldots,p,
\end{equation}
and
\textcolor{white}{xxxxxxxxxxxxxxxxxxxxxxxxxxxxxx}
\begin{align}\label{eq3.16}
\int_0^1S_0(\XXX;\tau;p)\,\dd\tau
&
=\int_0^1(S_0(\XXX;\tau;p)-S_p(\XXX;\tau;p))\,\dd\tau
\nonumber
\\
&
=\sum_{h=0}^{p-1}z^h\int_0^{z^{-h}}(S_h(\XXX;\tau;p)-S_{h+1}(\XXX;\tau;p))\,\dd\tau,
\end{align}
where
\textcolor{white}{xxxxxxxxxxxxxxxxxxxxxxxxxxxxxx}
\begin{equation}\label{eq3.17}
\int_0^{z^{-h}}(S_h(\XXX;\tau;p)-S_{h+1}(\XXX;\tau;p))\,\dd\tau\ge0,
\quad
h=0,1,\ldots,p-1.
\end{equation}
Furthermore, if $\XXX$ is $(A;M_1)$-anti-crowded, where $A\ge2$ and $1<M_1\le M$, then
\begin{equation}\label{eq3.18}
\int_0^1S_0(\XXX;\tau;p)\,\dd\tau\le\frac{A^2M^2}{z^p}.
\end{equation}
\end{lem}

\begin{proof}
Note from \eqref{eq3.14} that
\begin{displaymath}
S_0(\XXX;\tau;p)
=\sum_{s_1=1}^z\ldots\sum_{s_p=1}^z
\left\vert\vert(I(s_1,\ldots,s_p)-\tau)\cap\XXX\vert-\frac{\vert[0,1)\cap\XXX\vert}{z^p}\right\vert^2.
\end{displaymath}
Let the integer $h=0,1,\ldots,p$ be fixed.
For any fixed choice of $(s^*_1,\ldots,s^*_p)$ of the parameters $(s_1,\ldots,s_p)$, the collection
\begin{displaymath}
I((s^*_1,\ldots,s^*_h,s^*_{h+1},\ldots,s^*_p)-az^{-h},
\quad
a=0,1,\ldots,z^h-1,
\end{displaymath}
is a permutation of the collection
\begin{displaymath}
I((s_1,\ldots,s_h,s^*_{h+1},\ldots,s^*_p),
\quad
s_1,\ldots,s_h=1,\ldots,z.
\end{displaymath}
This implies that the the function $S_0(\XXX;\tau;p)$ is periodic in $\tau$ with period~$z^{-h}$, and the identity \eqref{eq3.15} follows immediately.

Next, the first equality in \eqref{eq3.16} is a consequence of the simple observation that $S_p(\XXX;\tau;p)=0$.
On the other hand, it is easy to check, corresponding to \eqref{eq3.11}, that for every $h=0,1,\ldots,p-1$,
\begin{align}\label{eq3.19}
&
S_h(\XXX;\tau;p)-S_{h+1}(\XXX;\tau;p)
\nonumber
\\
&\quad
=\frac{1}{z^{p-h-1}}\sum_{s_1=1}^z\ldots\sum_{s_{h+1}=1}^z
\left\vert\vert(I(s_1,\ldots,s_{h+1})-\tau)\cap\XXX\vert-\frac{\vert(I(s_1,\ldots,s_h)-\tau)\cap\XXX\vert}{z}\right\vert^2
\nonumber
\\
&\quad
\ge0.
\end{align}
The inequality \eqref{eq3.17} follows immediately.
Arguing as before, we observe that the function $S_h(\XXX;\tau;p)-S_{h+1}(\XXX;\tau;p)$ is periodic in $\tau$ with period~$z^{-h}$, so that
\begin{displaymath}
\int_0^1(S_h(\XXX;\tau;p)-S_{h+1}(\XXX;\tau;p))\,\dd\tau=z^h\int_0^{z^{-h}}(S_h(\XXX;\tau;p)-S_{h+1}(\XXX;\tau;p))\,\dd\tau.
\end{displaymath}
The second equality in \eqref{eq3.16} follows.

Finally, note that for every $\tau\in[0,1)$, we have $\vert I(s_1,\ldots,s_p)-\tau\vert=z^{-p}\ge M_1^{-1}$.
If $\XXX$ is $(A;M_1)$-anti-crowded, then
\begin{equation}\label{eq3.20}
\vert(I(s_1,\ldots,s_p)-\tau)\cap\XXX\vert\le\frac{AM}{z^p}.
\end{equation}
On the other hand, corresponding to \eqref{eq3.5}, we have
\begin{equation}\label{eq3.21}
S_0(\XXX;\tau;p)\le\sum_{s_1=1}^z\ldots\sum_{s_p=1}^z\vert(I(s_1,\ldots,s_p)-\tau)\cap\XXX\vert^2.
\end{equation}
The assertion \eqref{eq3.18} now follows on combining \eqref{eq3.20} and \eqref{eq3.21}.
\end{proof}

%
%

\section{Starting the proof of super-fast spreading}\label{sec4}

Suppose that a direction $\theta$ is bounded away from the directions of the edges of the translation surface~$\PPP$.
Recall from Section~\ref{sec2} that perpendicular projection of the edges of $\PPP$ to lines
perpendicular to the direction $\theta$ of the flow leads to a mapping
\begin{displaymath}
\psi:\EEE\to[0,1)
\end{displaymath}
that essentially gives a one-to-one correspondence between the points on the edges of $\PPP$ and points on the unit interval $[0,1)$,
apart from the finitely many singularities arising from the vertices of~$\PPP$.
There are some restrictions on~$\theta$, which we shall recall when appropriate.

With the help of this mapping, the interval exchange transformation generated by geodesic flow on $\PPP$ in the direction $\theta$
can be represented in the form
\begin{displaymath}
T_\theta:[0,1)\to[0,1).
\end{displaymath}
We have the following flow diagram, where the image of a subinterval $J_{i'}\subset E_{i'}$ under geodesic flow
in the direction $\theta$ on the first edge $E_{i''}$ that the flow encounters is a subinterval $J_{i''}\subset E_{i''}$.
\begin{displaymath}
\begin{picture}(327,88)(58,5)
\put(77,70){$\EEE\supset E_{i'}\supset J_{i'}$}
\put(78,10){$[0,1)\supset H_{i'}(J_{i'})$}
\put(290,70){$J_{i''}\subset E_{i''}\subset\EEE$}
\put(275,10){$H_{i''}(J_{i''})\subset[0,1)$}
\put(150,73){\vector(1,0){135}}
\put(161,13){\vector(1,0){109}}
\put(138,60){\vector(0,-1){36}}
\put(297,60){\vector(0,-1){36}}
\put(185,79){geodesic flow}
\put(183,61){in direction $\theta$}
\put(171,19){$T_\theta:[0,1)\to[0,1)$}
\put(305,40){$\psi_{i''}:E_{i''}\to H_{i''}$}
\put(60,40){$\psi_{i'}:E_{i'}\to H_{i'}$}
\end{picture}
\end{displaymath}
Clearly $T_\theta$ preserves $1$-dimensional Lebesgue measure.

Let the integer $N\ge1$ be fixed, and consider the finite geodesic segment $\LLL_\theta(t)$, $0\le t\le C_3N$, on the translation surface~$\PPP$.
Let
\begin{equation}\label{eq4.1}
0<t_1<t_2<\ldots<t_M
\end{equation}
denote the time instances when this finite geodesic segment intersects an edge of~$\PPP$.
Formally, we have
\textcolor{white}{xxxxxxxxxxxxxxxxxxxxxxxxxxxxxx}
\begin{equation}\label{eq4.2}
\LLL_\theta(t_j)\in\EEE,
\quad
j=1,\ldots,M,
\end{equation}
and these points are projected to the points
\begin{equation}\label{eq4.3}
x_j=\psi(\LLL_\theta(t_j))\in[0,1),
\quad
j=1,\ldots,M,
\end{equation}
giving rise to a set
\textcolor{white}{xxxxxxxxxxxxxxxxxxxxxxxxxxxxxx}
\begin{equation}\label{eq4.4}
\XXX=\{x_1,\ldots,x_M\}\subset[0,1).
\end{equation}
Using the interval exchange transformation, we see that
\begin{equation}\label{eq4.5}
\XXX=\{T_\theta^jx_1:j=0,1,\ldots,M-1\}.
\end{equation}
Thus \eqref{eq4.1}--\eqref{eq4.5} together contain all the information on the intersection points of the geodesic segment $\LLL(t)$, $0\le t\le C_3N$,
with the edges
of~$\PPP$.

\begin{lem}\label{lem4.1}
Suppose that $c_0\ge4$ and $\theta\not\in\Omega(n_0;c_0)$, where $n_0$ is a fixed integer.
Suppose further that the subinterval $I\subset[0,1)$ is the $H$-image of part of an edge of $\PPP$ and satisfies
\textcolor{white}{xxxxxxxxxxxxxxxxxxxxxxxxxxxxxx}
\begin{equation}\label{eq4.6}
\vert I\vert\ge\frac{1}{c_0^2c_12^{n_0+1}}.
\end{equation}
Then for every integer $M>\max\{2c_0^2c_1,c_6\}2^{n_0}$, where the constant $c_6=c_6(\PPP)$ is given in Lemma~\ref{lem2.3}, we have
\textcolor{white}{xxxxxxxxxxxxxxxxxxxxxxxxxxxxxx}
\begin{equation}\label{eq4.7}
\frac{\vert I\cap\XXX\vert}{M\vert I\vert}\le\frac{4c_0^2c_1}{c_6}.
\end{equation}
\end{lem}

\begin{proof}
We shall first establish the result when the condition \eqref{eq4.6} is replaced by the more restrictive condition
\textcolor{white}{xxxxxxxxxxxxxxxxxxxxxxxxxxxxxx}
\begin{equation}\label{eq4.8}
\frac{1}{c_0^2c_12^{n_0+1}}\le\vert I\vert<\frac{1}{c_0^2c_12^{n_0}}.
\end{equation}
Suppose that $I=\psi(Q_0R_0)$, where the interval $Q_0R_0$ lies on some edge $E_{i_0}$ of~$\PPP$.
Then the inequalities \eqref{eq4.8} imply the inequalities \eqref{eq2.9}, and so the hypotheses of Lemmas \ref{lem2.2} and~\ref{lem2.3} are satisfied.
Consider now the transportation process described in Lemmas \ref{lem2.2} and~\ref{lem2.3}.
In the terminology of Lemma~\ref{lem2.2}, we have $u+w\ge c_52^{n_0}$.
In the terminology of Lemma~\ref{lem2.3}, there exists a subset $\JJJ$ of $c_62^{n_0}$ consecutive integers such that $0\in\JJJ$ and the intervals
\begin{displaymath}
Q_jR_j,
\quad
j\in\JJJ,
\end{displaymath}
are pairwise disjoint.
This means that
\begin{displaymath}
\bigcup_{j\in\JJJ}T_\theta^jI
\end{displaymath}
is a disjoint union, and so
\begin{equation}\label{eq4.9}
M\ge\left\vert\left(\bigcup_{j\in\JJJ}T_\theta^jI\right)\cap\XXX\right\vert
=\sum_{j\in\JJJ}\left\vert T_\theta^jI\cap\XXX\right\vert
=\sum_{j\in\JJJ}\left\vert I\cap T_\theta^{-j}\XXX\right\vert.
\end{equation}
We now attempt to replace each summand $\vert I\cap T_\theta^{-j}\XXX\vert$ by $\vert I\cap\XXX\vert$.

Consider the set
\textcolor{white}{xxxxxxxxxxxxxxxxxxxxxxxxxxxxxx}
\begin{displaymath}
\XXX(\JJJ)=\bigcup_{j^*\in\JJJ}T_\theta^{-j^\ast}\XXX.
\end{displaymath}
Since the sets $T_\theta^jI$, $j\in\JJJ$, are pairwise disjoint, it follows that the sets
\begin{displaymath}
\XXX_j=\{x\in\XXX(\JJJ):x\in T_\theta^jI\},
\quad
j\in\JJJ,
\end{displaymath}
are pairwise disjoint.
We have $\XXX=\{x_1,\ldots,x_M\}$, and we can write
\begin{displaymath}
T_\theta^{-j}\XXX=\{x_{1-j},\ldots,x_{M-j}\}.
\end{displaymath}
Suppose first that $j$ is positive.
Then on replacing $\vert I\cap T_\theta^{-j}\XXX\vert$ by $\vert I\cap\XXX\vert$, we gain the contribution from the points
\begin{displaymath}
I\cap(\XXX_j\cap\{x_{M-j+1},\ldots,x_M\}).
\end{displaymath}
Thus the cumulative gain from those positive $j\in\JJJ$ comes from a subset of
\begin{displaymath}
I\cap\{x_{M-j^{(+)}+1},\ldots,x_M\},
\end{displaymath}
where $j^{(+)}=\max\{j:j\in\JJJ\}\ge0$.
It follows that
\begin{equation}\label{eq4.10}
\sum_{\substack{{j\in\JJJ}\\{j>0}}}\left\vert I\cap T_\theta^{-j}\XXX\right\vert
\ge\sum_{\substack{{j\in\JJJ}\\{j>0}}}\vert I\cap\XXX\vert-j^{(+)}.
\end{equation}
Suppose next that $j$ is negative.
Then on replacing $\vert I\cap T_\theta^{-j}\XXX\vert$ by $\vert I\cap\XXX\vert$, we gain the contribution from the points
\begin{displaymath}
I\cap(\XXX_j\cap\{x_1,\ldots,x_{-j}\}).
\end{displaymath}
Thus the cumulative gain from those negative $j\in\JJJ$ comes from a subset of
\begin{displaymath}
I\cap\{x_1,\ldots,x_{-j^{(-)}}\},
\end{displaymath}
where $j^{(-)}=\max\{j:j\in\JJJ\}\le0$.
It follows that
\begin{equation}\label{eq4.11}
\sum_{\substack{{j\in\JJJ}\\{j<0}}}\left\vert I\cap T_\theta^{-j}\XXX\right\vert
\ge\sum_{\substack{{j\in\JJJ}\\{j<0}}}\vert I\cap\XXX\vert+j^{(-)}.
\end{equation}
Combining \eqref{eq4.10} and \eqref{eq4.11}, we conclude that
\begin{equation}\label{eq4.12}
\sum_{j\in\JJJ}\left\vert I\cap T_\theta^{-j}\XXX\right\vert
\ge\sum_{j\in\JJJ}\left\vert I\cap\XXX\right\vert-\left(j^{(+)}-j^{(-)}\right)
\ge\sum_{j\in\JJJ}\left\vert I\cap\XXX\right\vert-c_62^{n_0}.
\end{equation}

It now follows from \eqref{eq4.9} and \eqref{eq4.12} that
\begin{equation}\label{eq4.13}
M\ge c_62^{n_0}\vert I\cap\XXX\vert-c_62^{n_0}.
\end{equation}
Suppose on the contrary that
\textcolor{white}{xxxxxxxxxxxxxxxxxxxxxxxxxxxxxx}
\begin{equation}\label{eq4.14}
\frac{\vert I\cap\XXX\vert}{M\vert I\vert}>\frac{4c_0^2c_1}{c_6}.
\end{equation}
Then combining that with \eqref{eq4.8} and \eqref{eq4.13}, we obtain
\begin{equation}\label{eq4.15}
M>2M-c_62^{n_0}.
\end{equation}
This and the inequality $M>c_62^{n_0}$ clearly lead to the absurdity
\begin{displaymath}
M>2M-c_62^{n_0}>M.
\end{displaymath}
Thus \eqref{eq4.14} cannot hold, and the desired inequality \eqref{eq4.7} follows immediately.

Consider now the general case where the interval $I$ satisfies the condition \eqref{eq4.6}.
Let $\mu$ be the unique integer satisfying
\begin{displaymath}
\frac{\mu}{c_0^2c_12^{n_0+1}}\le\vert I\vert<\frac{\mu+1}{c_0^2c_12^{n_0+1}}.
\end{displaymath}
Then we can write the interval $I$ as a pairwise disjoint union
\begin{displaymath}
I=I_1\cup\ldots\cup I_\mu,
\end{displaymath}
where
\textcolor{white}{xxxxxxxxxxxxxxxxxxxxxxxxxxxxxx}
\begin{displaymath}
\vert I_1\vert=\ldots=\vert I_{\mu-1}\vert=\frac{1}{c_0^2c_12^{n_0+1}}
\quad\mbox{and}\quad
\frac{1}{c_0^2c_12^{n_0+1}}\le\vert I_\mu\vert<\frac{1}{c_0^2c_12^{n_0}}.
\end{displaymath}
Applying the special case to each of $I_1,\ldots,I_\mu$, we conclude that
\begin{displaymath}
\frac{\vert I_i\cap\XXX\vert}{M}
\le\frac{4c_0^2c_1}{c_6}\vert I_i\vert,
\quad
i=1,\ldots,\mu,
\end{displaymath}
so that
\textcolor{white}{xxxxxxxxxxxxxxxxxxxxxxxxxxxxxx}
\begin{displaymath}
\frac{\vert I\cap\XXX\vert}{M}
=\sum_{i=1}^\mu\frac{\vert I_i\cap\XXX\vert}{M}
\le\frac{4c_0^2c_1}{c_6}\sum_{i=1}^\mu\vert I_i\vert
=\frac{4c_0^2c_1}{c_6}\vert I\vert,
\end{displaymath}
and so \eqref{eq4.7} follows again.
\end{proof}

We have therefore shown that if the integer $M>\max\{2c_0^2c_1,c_6\}2^{n_0}$,
then the set $\XXX$ given by \eqref{eq4.4} is $(A,M_1)$-anti-crowded, where
\begin{displaymath}
A=\frac{4c_0^2c_1}{c_6}
\quad\mbox{and}\quad
M_1=2c_0^2c_12^{n_0}.
\end{displaymath}

We shall assume that the parameter $z$ in Section~\ref{sec3} is an integer power of~$2$,
so that there exists a positive integer $z_1$ such that
\begin{equation}\label{eq4.16}
z=2^{z_1}.
\end{equation}
For any fixed $N\ge1$, not necessarily an integer, write
\begin{equation}\label{eq4.17}
N=2^{n_1}=(2^{z_1})^{n_1/z_1}=z^{n_1/z_1}.
\end{equation}
Furthermore, for computational simplicity, it is convenient to assume that $Nc_0^2c_1$ is a power of $z$ with integer exponent, so that
\begin{equation}\label{eq4.18}
Nc_0^2c_1=z^{n_1/z_1}c_0^2c_1=z^{n_1/z_1+c^*},
\end{equation}
where $n_1/z_1+c^*$ is an integer.

Consider the sets
\begin{equation}\label{eq4.19}
\Omega^c(n_1+iz_1;c_0)=[0,2\pi)\setminus\Omega(n_1+iz_1;c_0),
\quad
i=0,1,\ldots,k,
\end{equation}
where $k$ is a constant to be specified later.
It then follows from \eqref{eq2.7} that
\begin{equation}\label{eq4.20}
\lambda(\Omega^c(n_1+iz_1;c_0))\ge2\pi-\frac{8C^\star}{c_0},
\quad
i=0,1,\ldots,k,
\end{equation}
where $\lambda$ denotes $1$-dimensional Lebesgue measure.

For any $\eta\in(0,1)$, consider the set
\begin{equation}\label{eq4.21}
\Omega^c_N(\eta)=\{\theta\in[0,2\pi):\mbox{\eqref{eq4.22} and \eqref{eq4.23} hold}\},
\end{equation}
where
\textcolor{white}{xxxxxxxxxxxxxxxxxxxxxxxxxxxxxx}
\begin{equation}\label{eq4.22}
\theta\in\Omega^c(n_1+kz_1;c_0)
\end{equation}
and
\textcolor{white}{xxxxxxxxxxxxxxxxxxxxxxxxxxxxxx}
\begin{equation}\label{eq4.23}
\vert\{i=0,1,\ldots,k-2:\theta\in\Omega^c(n_1+iz_1;c_0)\}\vert\ge\eta(k-1).
\end{equation}
Thus $\Omega^c_N(\eta)$ is the set of values $\theta\in[0,2\pi)$ that are contained in $\Omega^c(n_1+kz_1;c_0)$
as well as at least $\eta(k-1)$ of the sets in \eqref{eq4.19} corresponding to $i=0,1,\ldots,k-2$.

To estimate the Lebesgue measure of $\Omega^c_N(\eta)$, we have the following result.

\begin{lem}\label{lem4.2}
Consider $U_1,\ldots,U_k\subset[0,2\pi)$ such that $\lambda(U_i)\ge2\pi-\delta$, $i=1,\ldots,k$,
where $\lambda$ denotes $1$-dimensional Lebesgue measure and $\delta\in(0,2\pi)$ is arbitrary.
For any $\eta\in(0,1)$, consider the set
\begin{displaymath}
\UUU(\eta)=\{x\in[0,2\pi):\vert\{i=1,\ldots,k:x\in U_i\}\vert\ge\eta k\}
\end{displaymath}
of values $x\in[0,2\pi)$ that are contained in at least $\eta k$ of the sets $U_1,\ldots,U_k$.
Then
\begin{equation}\label{eq4.24}
\lambda(\UUU(\eta))\ge2\pi-\delta-2\pi\eta.
\end{equation}
\end{lem}

\begin{proof}
For any $x\in[0,2\pi)$, let
\begin{displaymath}
\MMMM(x)
=\sum_{\substack{{i=1}\\{x\in U_i}}}^k1
=\vert\{i=1,\ldots,k:x\in U_i\}\vert.
\end{displaymath}
Then we clearly have
\textcolor{white}{xxxxxxxxxxxxxxxxxxxxxxxxxxxxxx}
\begin{equation}\label{eq4.25}
\int_0^{2\pi}\MMMM(x)\,\dd x
=\sum_{i=1}^k\lambda(U_i)
\ge k(2\pi-\delta).
\end{equation}
Let $\UUU^c(\eta)=[0,2\pi)\setminus\UUU(\eta)$ denote the complement of $\UUU(\eta)$.
Then
\begin{equation}\label{eq4.26}
\int_0^{2\pi}\MMMM(x)\,\dd x=\int_{\UUU(\eta)}\MMMM(x)\,\dd x+\int_{\UUU^c(\eta)}\MMMM(x)\,\dd x,
\end{equation}
and
\textcolor{white}{xxxxxxxxxxxxxxxxxxxxxxxxxxxxxx}
\begin{equation}\label{eq4.27}
\int_{\UUU^c(\eta)}\MMMM(x)\,\dd x<2\pi\eta k,
\end{equation}
so it follows on combining \eqref{eq4.25}--\eqref{eq4.27} that
\begin{equation}\label{eq4.28}
\int_{\UUU(\eta)}\MMMM(x)\,\dd x\ge k(2\pi-\delta-2\pi\eta).
\end{equation}
The assertion \eqref{eq4.24} now follows from \eqref{eq4.28} on observing that $\MMMM(x)\le k$ for every $x\in[0,2\pi)$.
\end{proof}

Combining \eqref{eq4.20}--\eqref{eq4.23} and Lemma~\ref{lem4.2}, we conclude that
\begin{equation}\label{eq4.29}
\lambda(\Omega^c_N(\eta))
\ge2\pi-\frac{8C^\star}{c_0}-\frac{8C^\star}{c_0}-2\pi\eta
\ge(1-\eps)2\pi,
\end{equation}
if we ensure that the conditions
\begin{equation}\label{eq4.30}
c_0\ge\frac{16C^\star}{\pi\eps}
\quad\mbox{and}\quad
\eta=\frac{\eps}{2}
\end{equation}
are satisfied.
We shall prove Theorem~\ref{thm1} with $\Gamma(\PPP;N;\eps)=\Omega^c_N(\eta)$.

Note from \eqref{eq4.21} that for any $\theta\in\Omega^c_N(\eta)$, there exists a sequence $i_1,\ldots,i_r$, where
\begin{equation}\label{eq4.31}
r\ge\eta(k-1)
\quad\mbox{and}\quad
0\le i_1<i_2<\ldots<i_r\le k-2,
\end{equation}
such that
\textcolor{white}{xxxxxxxxxxxxxxxxxxxxxxxxxxxxxx}
\begin{equation}\label{eq4.32}
\theta\not\in\Omega(n_1+i_tz_1;c_0),
\quad
t=1,\ldots,r.
\end{equation}
Accordingly, write
\textcolor{white}{xxxxxxxxxxxxxxxxxxxxxxxxxxxxxx}
\begin{displaymath}
N_t=2^{n_1+i_tz_1}=z^{n_1/z_1+i_t},
\quad
t=1,\ldots,r.
\end{displaymath}

Let
\textcolor{white}{xxxxxxxxxxxxxxxxxxxxxxxxxxxxxx}
\begin{displaymath}
p=\frac{n_1}{z_1}+c^*+i_r+1\le\frac{n_1}{z_1}+c^*+k,
\end{displaymath}
and let
\textcolor{white}{xxxxxxxxxxxxxxxxxxxxxxxxxxxxxx}
\begin{displaymath}
M_1
=2c_0^2c_12^{n_1+kz_1}
>z^{n_1/z_1+c^*+k}
\ge z^p.
\end{displaymath}

Recall from Section~\ref{sec3} that the set $\XXX$ is $(\delta;z)$-balanced relative to a special interval $I(s_1,\ldots,s_h)$ of length $z^{-h}$ if
\begin{displaymath}
-\frac{\delta M}{z^{h+1}}<\vert I(s_1,\ldots,s_h,s)\cap\XXX\vert-\frac{\vert I(s_1,\ldots,s_h)\cap\XXX\vert}{z}<\frac{\delta M}{z^{h+1}},
\quad
s=1,\ldots,z.
\end{displaymath}
We are interested in the cases when $h$ is of the form $n_1/z_1+c^*+i_t$, $t=1,\ldots,r$, for which \eqref{eq4.32} holds.
We distinguish two cases.

\begin{case1}
There exists a subset $\TTT_1\subset\{1,\ldots,r\}$ with cardinality
\begin{displaymath}
\vert\TTT_1\vert\ge\frac{r}{2}\ge\frac{\eta(k-1)}{2}
\end{displaymath}
such that for every integer $h$ of the form
\begin{equation}\label{eq4.33}
h=\frac{n_1}{z_1}+c^*+i_t,
\quad
t\in\TTT_1,
\end{equation}
for which \eqref{eq4.32} holds, there exists at least one integer sequence
\begin{displaymath}
(s_1,\ldots,s_h)\in\{1,\ldots,z\}^h
\end{displaymath}
of length $h$ such that the interval $I(s_1,\ldots,s_h)$ does not contain any singularity of the mapping $\psi:\EEE\to[0,1)$ arising from the vertices
of~$\PPP$, and the set $\XXX$ is not $(\delta;z)$-balanced relative to $I(s_1,\ldots,s_h)$.
\end{case1}

\begin{case2}
There exists a subset $\TTT_2\subset\{1,\ldots,r\}$ with cardinality
\begin{displaymath}
\vert\TTT_2\vert\ge\frac{r}{2}\ge\frac{\eta(k-1)}{2}
\end{displaymath}
such that for every integer $h$ of the form
\begin{equation}\label{eq4.34}
h=\frac{n_1}{z_1}+c^*+i_t,
\quad
t\in\TTT_2,
\end{equation}
for which \eqref{eq4.32} holds, and for every integer sequence
\begin{displaymath}
(s_1,\ldots,s_h)\in\{1,\ldots,z\}^h
\end{displaymath}
of length $h$ such that the interval $I(s_1,\ldots,s_h)$ does not contain any singularity of the mapping $\psi:\EEE\to[0,1)$
arising from the vertices of~$\PPP$, the set $\XXX$ is $(\delta;z)$-balanced relative to $I(s_1,\ldots,s_h)$.
\end{case2}

Before we study the two cases separately, we first establish some estimates that are common to both.

Let $h$ be a value given by \eqref{eq4.33} or \eqref{eq4.34}.

For any integer sequence $(s_1,\ldots,s_h)\in\{1,\ldots,z\}^h$ and any integer $s\in\{1,\ldots,z\}$, we need to consider the error
\begin{displaymath}
\left\vert\vert I(\bfs,s)\cap\XXX\vert-\frac{\vert I(\bfs)\cap\XXX\vert}{z}\right\vert,
\end{displaymath}
where $\bfs=(s_1,\ldots,s_h)$.

Consider the pre-image of the interval $I(\bfs)$ under the mapping $\psi:\EEE\to[0,1)$.
As $I(\bfs)$ does not contain any singularity arising from the vertices of~$\PPP$, the pre-image is an interval $Q_0R_0$ on some edge $E_{i_0}$
of~$\PPP$, with $H$-length
\begin{displaymath}
c_1z^{-h}
=\frac{c_1}{z^{n_1/z_1+c^*+i_t}}
=\frac{c_1}{c_0^2c_1z^{n_1/z_1+i_t}}
=\frac{1}{c_0^22^{n_1+z_1i_t}},
\end{displaymath}
in view of \eqref{eq4.16} and \eqref{eq4.18}.
Since \eqref{eq4.32} holds, we can consider the transportation process described in Lemma~\ref{lem2.3} in the special case $n_0=n_1+z_1i_t$.
Then for any subset $\JJJ\subset\{j\in\Zz:-u\le j\le w\}$ of $c_62^{n_1+z_1i_t}$ consecutive integers, the intervals
\begin{displaymath}
Q_jR_j,
\quad
j\in\JJJ,
\end{displaymath}
on the edges of $\PPP$ are pairwise disjoint.
Moving over from $\EEE$ to the interval $[0,1)$, we conclude that the subintervals
\begin{equation}\label{eq4.35}
T_\theta^jI(\bfs),
\quad
j\in\JJJ,
\end{equation}
are pairwise disjoint and of common length~$z^{-h}$.

\begin{lem}\label{lem4.3}
Let $\JJJ\subset\{j\in\Zz:-u\le j\le w\}$ be a subset of $c_62^{n_1+z_1i_t}$ consecutive integers including~$0$.
Then for any integer sequence $\bfs=(s_1,\ldots,s_h)\in\{1,\ldots,z\}^h$ and any integer $s\in\{1,\ldots,z\}$, we have
\begin{equation}\label{eq4.36}
\sum_{j\in\JJJ}\left\vert\vert T_\theta^jI(\bfs)\cap\XXX\vert-\vert I(\bfs)\cap\XXX\vert\right\vert\le2c_62^{n_1+z_1i_t}.
\end{equation}
as well as
\textcolor{white}{xxxxxxxxxxxxxxxxxxxxxxxxxxxxxx}
\begin{equation}\label{eq4.37}
\sum_{j\in\JJJ}\left\vert\vert T_\theta^jI(\bfs,s)\cap\XXX\vert-\vert I(\bfs,s)\cap\XXX\vert\right\vert\le2c_62^{n_1+z_1i_t}.
\end{equation}
\end{lem}

\begin{proof}
Write
\begin{displaymath}
\XXX^+=\XXX^+(\JJJ)=\bigcup_{j\in\JJJ}T_\theta^j\XXX
\quad\mbox{and}\quad
\XXX^-=\XXX^-(\JJJ)=\bigcap_{j\in\JJJ}T_\theta^j\XXX.
\end{displaymath}
Clearly
\textcolor{white}{xxxxxxxxxxxxxxxxxxxxxxxxxxxxxx}
\begin{displaymath}
\XXX^-\subset\XXX\subset\XXX^+,
\end{displaymath}
and
\textcolor{white}{xxxxxxxxxxxxxxxxxxxxxxxxxxxxxx}
\begin{equation}\label{eq4.38}
\vert\XXX^+\setminus\XXX^-\vert\le2c_62^{n_1+z_1i_t}.
\end{equation}
There are two possibilities.

If $\vert T_\theta^jI(\bfs)\cap\XXX\vert>\vert I(\bfs)\cap\XXX\vert$, then
\begin{align}\label{eq4.39}
\vert T_\theta^jI(\bfs)\cap\XXX^+\vert
&
\ge\vert T_\theta^jI(\bfs)\cap\XXX\vert
>\vert I(\bfs)\cap\XXX\vert
\nonumber
\\
&
=\vert T_\theta^jI(\bfs)\cap T_\theta^j\XXX\vert
\ge\vert T_\theta^jI(\bfs)\cap\XXX^-\vert.
\end{align}
On the other hand, if $\vert T_\theta^jI(\bfs)\cap\XXX\vert\le\vert I(\bfs)\cap\XXX\vert$, then
\begin{align}\label{eq4.40}
\vert T_\theta^jI(\bfs)\cap\XXX^-\vert
&
\le\vert T_\theta^jI(\bfs)\cap\XXX\vert
\le\vert I(\bfs)\cap\XXX\vert
\nonumber
\\
&
=\vert T_\theta^jI(\bfs)\cap T_\theta^j\XXX\vert
\le\vert T_\theta^jI(\bfs)\cap\XXX^+\vert.
\end{align}
In either case, it follows from \eqref{eq4.39} and \eqref{eq4.40} that
\begin{equation}\label{eq4.41}
\left\vert\vert T_\theta^jI(\bfs)\cap\XXX\vert-\vert I(\bfs)\cap\XXX\vert\right\vert
\le\vert T_\theta^jI(\bfs)\cap\XXX^+\vert-\vert T_\theta^jI(\bfs)\cap\XXX^-\vert.
\end{equation}

Next, since the sets in \eqref{eq4.35} are pairwise disjoint, we have
\begin{align}\label{eq4.42}
&
\sum_{j\in\JJJ}\left(\vert T_\theta^jI(\bfs)\cap\XXX^+\vert-\vert T_\theta^jI(\bfs)\cap\XXX^-\vert\right)
=\sum_{j\in\JJJ}\vert T_\theta^jI(\bfs)\cap(\XXX^+\setminus\XXX^-)\vert
\nonumber
\\
&\quad
=\left\vert\left(\bigcup_{j\in\JJJ}T_\theta^jI(\bfs)\right)\cap(\XXX^+\setminus\XXX^-)\right\vert
\le\vert\XXX^+\setminus\XXX^-\vert.
\end{align}
The estimate \eqref{eq4.36} now follows on combining \eqref{eq4.38}, \eqref{eq4.41} and \eqref{eq4.42}.

The estimate \eqref{eq4.37} can be established in a similar way.
\end{proof}

%
%

\section{Studying Case 1}\label{sec5}

In this section, we consider the situation when Case~1 holds.

Consider any value $h$ given by \eqref{eq4.33}.
It follows from the hypotheses of this case that there exist an integer sequence $\bfs=(s_1,\ldots,s_h)\in\{1,\ldots,z\}^h$
and an integer $s\in\{1,\ldots,z\}$ such that
\begin{equation}\label{eq5.1}
\left\vert\vert I(\bfs,s)\cap\XXX\vert-\frac{\vert I(\bfs)\cap\XXX\vert}{z}\right\vert\ge\frac{\delta M}{z^{h+1}}.
\end{equation}

Using a routine averaging argument, it follows from \eqref{eq4.36} that there exists a subset $\JJJ_1\subset\JJJ$ such that
$\vert\JJJ_1\vert\ge3c_62^{n_1+z_1i_t}/4$ and the inequality
\begin{equation}\label{eq5.2}
\left\vert\vert T_\theta^jI(\bfs)\cap\XXX\vert-\vert I(\bfs)\cap\XXX\vert\right\vert\le8
\end{equation}
holds for every $j\in\JJJ_1$.
Similarly, it follows from \eqref{eq4.37} that there exists a subset $\JJJ_2\subset\JJJ$
such that $\vert\JJJ_2\vert\ge3c_62^{n_1+z_1i_t}/4$ and the inequality
\begin{equation}\label{eq5.3}
\left\vert\vert T_\theta^jI(\bfs,s)\cap\XXX\vert-\vert I(\bfs,s)\cap\XXX\vert\right\vert\le8
\end{equation}
holds for every $j\in\JJJ_2$.
Now let $\JJJ^*=\JJJ_1\cap\JJJ_2$.
Then clearly
\begin{equation}\label{eq5.4}
\vert\JJJ^*\vert\ge\frac{c_62^{n_1+z_1i_t}}{2}.
\end{equation}
Furthermore, the inequalities \eqref{eq5.2} and \eqref{eq5.3} both hold for every $j\in\JJJ^*$.

Combining \eqref{eq5.2} and \eqref{eq5.3}, we see that the inequality
\begin{displaymath}
\left\vert\vert I(\bfs,s)\cap\XXX\vert-\frac{\vert I(\bfs)\cap\XXX\vert}{z}\right\vert
\le\left\vert\vert T_\theta^jI(\bfs,s)\cap\XXX\vert-\frac{\vert T_\theta^jI(\bfs)\cap\XXX\vert}{z}\right\vert+8+\frac{8}{z}
\end{displaymath}
holds for every $j\in\JJJ^*$.
Combining this with \eqref{eq5.1}, we deduce that the inequality
\begin{equation}\label{eq5.5}
\left\vert\vert T_\theta^jI(\bfs,s)\cap\XXX\vert-\frac{\vert T_\theta^jI(\bfs)\cap\XXX\vert}{z}\right\vert
\ge\frac{\delta M}{z^{h+1}}-8-\frac{8}{z}
\ge\frac{3\delta M}{4z^{h+1}}
\end{equation}
holds for every $j\in\JJJ^*$, provided that
\begin{equation}\label{eq5.6}
\delta M\ge64z^{h+1}.
\end{equation}

Motivated by an average version of \eqref{eq3.19} where the parameter $\tau$ runs over an interval $[0,z^{-h}]$, we consider the following.
Each of the disjoint intervals
\begin{displaymath}
T_\theta^jI(\bfs),
\quad
j\in\JJJ^*,
\end{displaymath}
is a subinterval of $[0,1)$ of length $z^{-h}$.
Hence there exist a unique integer sequence $\bfs^{(j)}=(s_1^{(j)},\ldots,s_h^{(j)})\in\{1,\ldots,z\}^h$
and a unique real number $\tau_j\in[0,z^{-h})$ such that
\begin{align}
T_\theta^jI(\bfs)
&
=I(\bfs^{(j)})-\tau_j,
\label{eq5.7}
\\
T_\theta^jI(\bfs,s)
&
=I(\bfs^{(j)},s)-\tau_j,
\nonumber
\end{align}
so that
\textcolor{white}{xxxxxxxxxxxxxxxxxxxxxxxxxxxxxx}
\begin{align}
T_\theta^jI(\bfs)\cap\XXX
&
=(I(\bfs^{(j)})-\tau_j)\cap\XXX,
\label{eq5.8}
\\
T_\theta^jI(\bfs,s)\cap\XXX
&
=(I(\bfs^{(j)},s)-\tau_j)\cap\XXX.
\label{eq5.9}
\end{align}
Note that on the right hand sides of \eqref{eq5.8} and \eqref{eq5.9}, there is a fixed shift~$\tau_j$.
We shall now replace it by a shift~$\tau$, and let $\tau$ run over a short interval centered at~$\tau_j$.

Recall that if
\textcolor{white}{xxxxxxxxxxxxxxxxxxxxxxxxxxxxxx}
\begin{equation}\label{eq5.10}
M>\max\{2c_0^2c_1,c_6\}2^{n_1+kz_1},
\end{equation}
then it follows from Lemma~\ref{lem4.1} that the $M$-element set $\XXX\subset[0,1)$ is $(A,M_1)$-anti-crowded with
\textcolor{white}{xxxxxxxxxxxxxxxxxxxxxxxxxxxxxx}
\begin{equation}\label{eq5.11}
A=\frac{4c_0^2c_1}{c_6}
\quad\mbox{and}\quad
M_1=2c_0^2c_12^{n_1+kz_1},
\end{equation}
so that for every subinterval $I\subset[0,1)$ satisfying the restriction $\vert I\vert\ge M_1^{-1}$, we have
\begin{equation}\label{eq5.12}
\vert I\cap\XXX\vert\le AM\vert I\vert.
\end{equation}

Consider the two intervals $T_\theta^jI(\bfs)$ and $I(\bfs^{(j)})-\tau$, and suppose that $\vert\tau-\tau_j\vert$ is sufficiently small that
the two intervals overlap.
In view of \eqref{eq5.7}, their symmetric difference
\textcolor{white}{xxxxxxxxxxxxxxxxxxxxxxxxxxxxxx}
\begin{displaymath}
T_\theta^jI(\bfs)\,\triangle\,(I(\bfs^{(j)})-\tau)=\III_1\cup\III_2
\end{displaymath}
is a union of two intervals of length $\vert\tau-\tau_j\vert$, and
\begin{equation}\label{eq5.13}
\left\vert\vert T_\theta^jI(\bfs)\cap\XXX\vert-\vert(I(\bfs^{(j)})-\tau)\cap\XXX\vert\right\vert
\le\max\{\vert\III_1\cap\XXX\vert,\vert\III_2\cap\XXX\vert\}.
\end{equation}
For the case $\vert\tau-\tau_j\vert\ge M_1^{-1}$, it follows from \eqref{eq5.12} and \eqref{eq5.13} that
\begin{equation}\label{eq5.14}
\left\vert\vert T_\theta^jI(\bfs)\cap\XXX\vert-\vert(I(\bfs^{(j)})-\tau)\cap\XXX\vert\right\vert
\le AM\vert\tau-\tau_j\vert.
\end{equation}
For the alternative case $\vert\tau-\tau_j\vert<M_1^{-1}$, we may assume, without loss of generality,
that $\vert\III_1\cap\XXX\vert\ge\vert\III_2\cap\XXX\vert$, and let $\III_0$ be an interval
such that $\vert\III_0\vert=M_1^{-1}$ and $\III_1\subset\III_0$.
Then it follows from this and \eqref{eq5.13} that
\begin{equation}\label{eq5.15}
\left\vert\vert T_\theta^jI(\bfs)\cap\XXX\vert-\vert(I(\bfs^{(j)})-\tau)\cap\XXX\vert\right\vert
\le\vert\III_0\cap\XXX\vert.
\end{equation}
Combining \eqref{eq5.12} and \eqref{eq5.15}, we deduce that
\begin{equation}\label{eq5.16}
\left\vert\vert T_\theta^jI(\bfs)\cap\XXX\vert-\vert(I(\bfs^{(j)})-\tau)\cap\XXX\vert\right\vert
\le AMM_1^{-1}.
\end{equation}
Finally it follows from \eqref{eq5.14} and \eqref{eq5.16} that
\begin{equation}\label{eq5.17}
\left\vert\vert T_\theta^jI(\bfs)\cap\XXX\vert-\vert(I(\bfs^{(j)})-\tau)\cap\XXX\vert\right\vert
\le AM\max\{\vert\tau-\tau_j\vert,M_1^{-1}\}.
\end{equation}

Suppose that
\textcolor{white}{xxxxxxxxxxxxxxxxxxxxxxxxxxxxxx}
\begin{equation}\label{eq5.18}
\vert\tau-\tau_j\vert\le\max\left\{\frac{\delta}{8Az^{h+1}},\frac{1}{M_1}\right\}=\frac{\delta}{8Az^{h+1}},
\end{equation}
under the extra restriction
\textcolor{white}{xxxxxxxxxxxxxxxxxxxxxxxxxxxxxx}
\begin{equation}\label{eq5.19}
\delta M_1\ge8Az^{h+1}.
\end{equation}
Then it follows from \eqref{eq5.17}--\eqref{eq5.19} that
\begin{equation}\label{eq5.20}
\left\vert\vert T_\theta^jI(\bfs)\cap\XXX\vert-\vert(I(\bfs^{(j)})-\tau)\cap\XXX\vert\right\vert
\le\frac{\delta M}{8z^{h+1}},
\end{equation}
and analogous argument gives
\begin{equation}\label{eq5.21}
\left\vert\vert T_\theta^jI(\bfs,s)\cap\XXX\vert-\vert(I(\bfs^{(j)},s)-\tau)\cap\XXX\vert\right\vert
\le\frac{\delta M}{8z^{h+1}},
\end{equation}
assuming that the two intervals $T_\theta^jI(\bfs,s)$ and $I(\bfs^{(j)},s)-\tau$ overlap.
Combining \eqref{eq5.20} and \eqref{eq5.21}, we see that the inequality
\begin{align}\label{eq5.22}
&
\left\vert\vert T_\theta^jI(\bfs,s)\cap\XXX\vert-\frac{\vert T_\theta^jI(\bfs)\cap\XXX\vert}{z}\right\vert
\nonumber
\\
&\quad
\le\left\vert\vert(I(\bfs^{(j)},s)-\tau)\cap\XXX\vert-\frac{\vert(I(\bfs^{(j)})-\tau)\cap\XXX\vert}{z}\right\vert
+\frac{\delta M}{8z^{h+1}}+\frac{\delta M}{8z^{h+2}}
\end{align}
holds for every $j\in\JJJ^*$, assuming that \eqref{eq5.18} and \eqref{eq5.19} hold.
This, together with \eqref{eq5.5}, implies that the inequality
\begin{equation}\label{eq5.23}
\left\vert\vert(I(\bfs^{(j)},s)-\tau)\cap\XXX\vert-\frac{\vert(I(\bfs^{(j)})-\tau)\cap\XXX\vert}{z}\right\vert
\ge\frac{3\delta M}{4z^{h+1}}-\frac{\delta M}{4z^{h+1}}
=\frac{\delta M}{2z^{h+1}}
\end{equation}
holds for every $j\in\JJJ^*$, assuming that \eqref{eq5.6} holds also.

For brevity, let
\textcolor{white}{xxxxxxxxxxxxxxxxxxxxxxxxxxxxxx}
\begin{equation}\label{eq5.24}
\omega=\frac{\delta}{8Az^{h+1}}
\end{equation}
denote the upper bound in \eqref{eq5.18}.
We now return to the identity \eqref{eq3.19}.
Averaging $\tau$ over the interval $[0,z_h)$, we obtain the lower bound
\begin{align}\label{eq5.25}
&
z^h\int_0^{z^{-h}}(S_h(\XXX;\tau;p)-S_{h+1}(\XXX;\tau;p))\,\dd\tau
\nonumber
\\
&\quad
\ge\frac{z^h}{z^{p-h-1}}\sum_{j\in\JJJ^*}\sum_{s=1}^z\int_{\tau_j-\omega}^{\tau_j+\omega}
\left\vert\vert(I(\bfs^{(j)},s)-\tau)\cap\XXX\vert-\frac{\vert(I(\bfs^{(j)})-\tau)\cap\XXX\vert}{z}\right\vert^2\dd\tau
\nonumber
\\
&\quad
\ge\frac{2z^{h+1}\vert\JJJ^*\vert\omega}{z^{p-h-1}}\left(\frac{\delta M}{2z^{h+1}}\right)^2
\ge\frac{c_6\delta^3M^2}{32c_0^2c_1Az^{p+1}},
\end{align}
in view of \eqref{eq4.18}, \eqref{eq4.33}, \eqref{eq5.4}, \eqref{eq5.23} and \eqref{eq5.24}.

Recall that we are considering Case~1 here.
There exists a set $\TTT_1$ with cardinality $\vert\TTT_1\vert\ge\eta(k-1)/2$
such that for each of the $\vert\TTT_1\vert$ integers $h$ of the form
\eqref{eq4.33}, we have an estimate of the form \eqref{eq5.25}.
Combining \eqref{eq3.16}, \eqref{eq3.18} and \eqref{eq5.25}, we deduce that
\begin{align}\label{eq5.26}
A^2M^2
&
\ge z^p\int_0^1S_0(\XXX;\tau;p)\,\dd\tau
=z^p\sum_{h=0}^{p-1}z^h\int_0^{z^{-h}}(S_h(\XXX;\tau;p)-S_{h+1}(\XXX;\tau;p))\,\dd\tau
\nonumber
\\
&
\ge\vert\TTT_1\vert\frac{c_6\delta^3M^2}{32c_0^2c_1Az}
\ge\frac{\eta(k-1)c_6\delta^3M^2}{64c_0^2c_1Az},
\end{align}
so that
\textcolor{white}{xxxxxxxxxxxxxxxxxxxxxxxxxxxxxx}
\begin{equation}\label{eq5.27}
k-1\le\frac{64c_0^2c_1A^3z}{c_6\eta\delta^3}.
\end{equation}
Thus choosing $k$ to be an integer greater than the right hand side of \eqref{eq5.27} then ensures that Case~1 is impossible.

%
%

\section{Studying Case 2}\label{sec6}

In this section, we consider the situation when Case~2 holds.

Consider any value $h$ given by \eqref{eq4.34}.
It follows from the hypotheses of this case that for every integer sequence $\bfs=(s_1,\ldots,s_h)\in\{1,\ldots,z\}^h$
and every integer $s\in\{1,\ldots,z\}$,
\textcolor{white}{xxxxxxxxxxxxxxxxxxxxxxxxxxxxxx}
\begin{equation}\label{eq6.1}
\left\vert\vert I(\bfs,s)\cap\XXX\vert-\frac{\vert I(\bfs)\cap\XXX\vert}{z}\right\vert<\frac{\delta M}{z^{h+1}}.
\end{equation}

Let
\textcolor{white}{xxxxxxxxxxxxxxxxxxxxxxxxxxxxxx}
\begin{displaymath}
I(\bfs')=I(s'_1,\ldots,s'_h)
\quad\mbox{and}\quad
I(\bfs'')=I(s''_1,\ldots,s''_h)
\end{displaymath}
be distinct singularity free intervals.

Recall from \eqref{eq4.16} that $z=2^{z_1}$, where $z_1$ is a positive integer.
We may assume that $z_1\ge2$, so that $z$ is a multiple of~$4$.
We can therefore divide the interval $I(\bfs')$ into $4$ equal parts, and denote by $I^*(\bfs')$ the union of the two middle parts.
Likewise we can divide the interval $I(\bfs'')$ into $4$ equal parts, and denote by $I^*(\bfs'')$ the union of the two middle parts.

\begin{lem}\label{lem6.1}
Let $\JJJ\subset\{j\in\Zz:-u\le j\le w\}$ be a subset of $c_62^{n_1+z_1i_t}$ consecutive integers including~$0$.
Suppose that the inequality \eqref{eq6.1} holds for $\bfs',\bfs''\in\{1,\ldots,z\}^h$ and any $s\in\{1,\ldots,z\}$, and the inequality
\begin{equation}\label{eq6.2}
\frac{6r\delta M}{z^h}\le\vert I(\bfs')\cap\XXX\vert-\vert I(\bfs'')\cap\XXX\vert<\frac{6(r+1)\delta M}{z^h}
\end{equation}
holds for some integer $r\ge1$.
Then the set
\begin{equation}\label{eq6.3}
\WWW(\bfs',\bfs'')
=\left(\bigcup_{j\in\JJJ}T_\theta^jI^*(\bfs')\right)\cap\left(\bigcup_{j\in\JJJ}T_\theta^jI^*(\bfs'')\right)
\end{equation}
has $1$-dimensional Lebesgue measure
\begin{equation}\label{eq6.4}
\lambda(\WWW(\bfs',\bfs''))\le\frac{2c_62^{n_1+z_1i_t}z^2}{(z-4)(3r-1)\delta M-4(A+\delta)M}.
\end{equation}
\end{lem}

\begin{proof}
Suppose that for $j',j''\in\JJJ$, the intersection
\begin{equation}\label{eq6.5}
T_\theta^{j'}I^*(\bfs')\cap T_\theta^{j''}I^*(\bfs'')\ne\emptyset.
\end{equation}
Then it is easy to see that
\textcolor{white}{xxxxxxxxxxxxxxxxxxxxxxxxxxxxxx}
\begin{equation}\label{eq6.6}
\lambda(T_\theta^{j'}I(\bfs')\cap T_\theta^{j''}I(\bfs''))\ge\frac{1}{2z^h}.
\end{equation}
Let
\textcolor{white}{xxxxxxxxxxxxxxxxxxxxxxxxxxxxxx}
\begin{align}
\SSS'_-
&
=\{s\in\{1,\ldots,z\}:T_\theta^{j'}I(\bfs',s)\subset T_\theta^{j'}I(\bfs')\cap T_\theta^{j''}I(\bfs'')\},
\nonumber
\\
\SSS''_+
&
=\{s\in\{1,\ldots,z\}:T_\theta^{j''}I(\bfs'',s)\cap(T_\theta^{j'}I(\bfs')\cap T_\theta^{j''}I(\bfs''))\ne\emptyset\}.
\nonumber
\end{align}
It then follows from \eqref{eq6.6} that
\begin{equation}\label{eq6.7}
\frac{z}{2}-2\le\vert\SSS'_-\vert\le\vert\SSS''_+\vert\le\vert\SSS'_-\vert+4.
\end{equation}
On the other hand,
\begin{displaymath}
\bigcup_{s\in\SSS'_-}\left(T_\theta^{j'}I(\bfs',s)\cap\XXX\right)
\subset T_\theta^{j'}I(\bfs')\cap T_\theta^{j''}I(\bfs'')\cap\XXX
\subset\bigcup_{s\in\SSS''_+}\left(T_\theta^{j''}I(\bfs'',s)\cap\XXX\right).
\end{displaymath}
Using this and the triangle inequality, we deduce that
\begin{align}
&
\sum_{s\in\SSS'_-}\left(\vert I(\bfs',s)\cap\XXX\vert
-\left\vert\left\vert T_\theta^{j'}I(\bfs',s)\cap\XXX\right\vert-\vert I(\bfs',s)\cap\XXX\vert\right\vert\right)
\nonumber
\\
&\quad
\le\left\vert T_\theta^{j'}I(\bfs')\cap T_\theta^{j''}I(\bfs'')\cap\XXX\right\vert
\nonumber
\\
&\quad
\le\sum_{s\in\SSS''_+}\left(\vert I(\bfs'',s)\cap\XXX\vert
+\left\vert\left\vert T_\theta^{j''}I(\bfs'',s)\cap\XXX\right\vert-\vert I(\bfs'',s)\cap\XXX\vert\right\vert\right),
\nonumber
\end{align}
and so
\textcolor{white}{xxxxxxxxxxxxxxxxxxxxxxxxxxxxxx}
\begin{align}\label{eq6.8}
&
\sum_{s\in\SSS'_-}\vert I(\bfs',s)\cap\XXX\vert
-\sum_{s\in\SSS''_+}\vert I(\bfs'',s)\cap\XXX\vert
\nonumber
\\
&\quad
\le\sum_{s\in\SSS'_-}\left\vert\left\vert T_\theta^{j'}I(\bfs',s)\cap\XXX\right\vert-\vert I(\bfs',s)\cap\XXX\vert\right\vert
\nonumber
\\
&\quad\qquad
+\sum_{s\in\SSS''_+}\left\vert\left\vert T_\theta^{j''}I(\bfs'',s)\cap\XXX\right\vert-\vert I(\bfs'',s)\cap\XXX\vert\right\vert.
\end{align}
It clearly follows form \eqref{eq6.1} and the last inequality in \eqref{eq6.7} that
\begin{align}\label{eq6.9}
&
\sum_{s\in\SSS'_-}\vert I(\bfs',s)\cap\XXX\vert
-\sum_{s\in\SSS''_+}\vert I(\bfs'',s)\cap\XXX\vert
\nonumber
\\
&\quad
\ge\vert\SSS'_-\vert\left(\frac{\vert I(\bfs')\cap\XXX\vert}{z}-\frac{\delta M}{z^{h+1}}\right)
-\vert\SSS''_+\vert\left(\frac{\vert I(\bfs'')\cap\XXX\vert}{z}+\frac{\delta M}{z^{h+1}}\right)
\nonumber
\\
&\quad
\ge\vert\SSS'_-\vert\left(\frac{\vert I(\bfs')\cap\XXX\vert}{z}-\frac{\vert I(\bfs'')\cap\XXX\vert}{z}-\frac{2\delta M}{z^{h+1}}\right)
-\frac{4\vert I(\bfs'')\cap\XXX\vert}{z}-\frac{4\delta M}{z^{h+1}}.
\end{align}
Consider the last line in \eqref{eq6.9}.
For the first term, we apply the first inequality in \eqref{eq6.7} to $\vert\SSS'_-\vert$
and the first inequality in \eqref{eq6.2} to the first two terms inside the brackets.
For the second term, we apply $(A;M_1)$-anti-crowdedness of~$\XXX$.
Then
\begin{equation}\label{eq6.10}
\sum_{s\in\SSS'_-}\vert I(\bfs',s)\cap\XXX\vert-\sum_{s\in\SSS''_+}\vert I(\bfs'',s)\cap\XXX\vert
\ge\left(\frac{z}{2}-2\right)\frac{(6r-2)\delta M}{z^{h+1}}-\frac{4(A+\delta)M}{z^{h+1}}.
\end{equation}
Combining \eqref{eq6.8} and \eqref{eq6.10}, we have
\begin{align}
&\sum_{s\in\SSS'_-}\left\vert\left\vert T_\theta^{j'}I(\bfs',s)\cap\XXX\right\vert-\vert I(\bfs',s)\cap\XXX\vert\right\vert
+\sum_{s\in\SSS''_+}\left\vert\left\vert T_\theta^{j''}I(\bfs'',s)\cap\XXX\right\vert-\vert I(\bfs'',s)\cap\XXX\vert\right\vert
\nonumber
\\
&\quad
\ge\frac{(z-4)(3r-1)\delta M-4(A+\delta)M}{z^{h+1}}.
\nonumber
\end{align}
Note next that if \eqref{eq6.5} holds, then for any $j\in\JJJ$,
\begin{displaymath}
\begin{array}{ll}
T_\theta^{j'}I^*(\bfs')\cap T_\theta^jI^*(\bfs'')=\emptyset,
&\mbox{if $j\ne j''$},
\vspace{2pt}\\
T_\theta^jI^*(\bfs')\cap T_\theta^{j''}I^*(\bfs'')=\emptyset,
&\mbox{if $j\ne j'$}.
\end{array}
\end{displaymath}
Now let
\textcolor{white}{xxxxxxxxxxxxxxxxxxxxxxxxxxxxxx}
\begin{displaymath}
\JJJ(\bfs',\bfs'')=\{(j',j'')\in\JJJ\times\JJJ:T_\theta^{j'}I^*(\bfs')\cap T_\theta^{j''}I^*(\bfs'')\ne\emptyset\}.
\end{displaymath}
Then for each $j'\in\JJJ$, there is at most one $j''\in\JJJ$ for which $(j',j'')\in\JJJ(\bfs',\bfs'')$,
and for each $j''\in\JJJ$, there is at most one $j'\in\JJJ$ for which $(j',j'')\in\JJJ(\bfs',\bfs'')$.
It follows that
\textcolor{white}{xxxxxxxxxxxxxxxxxxxxxxxxxxxxxx}
\begin{align}
&
\frac{(z-4)(3r-1)\delta M-4(A+\delta)M}{z^{h+1}}\vert\JJJ(\bfs',\bfs'')\vert
\nonumber
\\
&\quad
\le\sum_{j'\in\JJJ}\sum_{s\in\SSS'_-}\left\vert\left\vert T_\theta^{j'}I(\bfs',s)\cap\XXX\right\vert-\vert I(\bfs',s)\cap\XXX\vert\right\vert
\nonumber
\\
&\quad\qquad
+\sum_{j''\in\JJJ}\sum_{s\in\SSS''_+}\left\vert\left\vert T_\theta^{j''}I(\bfs'',s)\cap\XXX\right\vert-\vert I(\bfs'',s)\cap\XXX\vert\right\vert.
\nonumber
\end{align}
Combining this with \eqref{eq4.37}, we conclude that
\begin{displaymath}
\vert\JJJ(\bfs',\bfs'')\vert\le\frac{4c_62^{n_1+z_1i_t}z^{h+2}}{(z-4)(3r-1)\delta M-4(A+\delta)M}.
\end{displaymath}
The inequality \eqref{eq6.4} now follows on combining this with the trivial observation that
\begin{displaymath}
\lambda(T_\theta^{j'}I^*(\bfs')\cap T_\theta^{j''}I^*(\bfs''))\le\frac{1}{2z^h}
\end{displaymath}
whenever \eqref{eq6.5} holds.
\end{proof}

Recall that the mapping $\psi:\EEE\to[0,1)$ takes any edge $E$ of $\PPP$, apart possibly from the two endpoints,
to a subinterval $\psi(E)$ of $[0,1)$.
We now consider the closed interval $\overline{\psi(E)}\subset[0,1)$ comprising the interval $\psi(E)$ together with its two endpoints,
and consider the set
\textcolor{white}{xxxxxxxxxxxxxxxxxxxxxxxxxxxxxx}
\begin{displaymath}
\frakI(E)=\{I(\bfs)\subset\overline{\psi(E)}:\bfs\in\{1,\ldots,z\}^h\}
\end{displaymath}
of all special intervals of length $z^{-h}$ in $\overline{\psi(E)}$.

\begin{lem}\label{lem6.2}
For any edge $E$ of~$\PPP$, at least one of the following two possibilities must hold:

\emph{(i)}
There exist two adjacent intervals $I(\bfs^*),I(\bfs^{**})\in\frakI(E)$ such that
\begin{equation}\label{eq6.11}
\vert I(\bfs^*)\cap\XXX\vert-\vert I(\bfs^{**})\cap\XXX\vert\ge\frac{3\delta M}{z^h}.
\end{equation}

\emph{(ii)}
There exists a constant $c_8=c_8(\PPP)>0$, depending at most on~$\PPP$, such that for any two intervals $I(\bfs'),I(\bfs'')\in\frakI(E)$,
\begin{equation}\label{eq6.12}
\vert\vert I(\bfs')\cap\XXX\vert-\vert I(\bfs'')\cap\XXX\vert\vert\le\frac{c_8\delta M}{z^h}.
\end{equation}
\end{lem}

\begin{proof}
Let $I(\bfs^{(1)}),I(\bfs^{(2)}),\ldots,I(\bfs^{(q)})\in\frakI(E)$ be a maximal chain of intervals such that
\textcolor{white}{xxxxxxxxxxxxxxxxxxxxxxxxxxxxxx}
\begin{equation}\label{eq6.13}
\vert I(\bfs^{(\rho)})\cap\XXX\vert-\vert I(\bfs^{(\rho+1)})\cap\XXX\vert\ge\frac{6\delta M}{z^h},
\quad
\rho=1,\ldots,q-1.
\end{equation}
If $q=1$, then \eqref{eq6.12} holds with $c_8=6$.
Thus we assume that $q\ge2$.

We have the trivial bounds
\begin{equation}\label{eq6.14}
1\ge\lambda\left(\bigcup_{\rho=1}^q\bigcup_{j\in\JJJ}T_\theta^jI^*(\bfs^{(\rho)})\right)
\ge\sum_{\rho=1}^q\sum_{j\in\JJJ}\lambda\left(T_\theta^jI^*(\bfs^{(\rho)})\right)-\Upsilon,
\end{equation}
where $\JJJ\subset\{j\in\Zz:-u\le j\le w\}$ is a subset of $c_62^{n_1+z_1i_t}$ consecutive integers as in Lemma~\ref{lem6.1}, and
\textcolor{white}{xxxxxxxxxxxxxxxxxxxxxxxxxxxxxx}
\begin{equation}\label{eq6.15}
\Upsilon=\sum_{\rho_1=1}^{q-1}\sum_{\rho_2=\rho_1+1}^q\lambda\left(\WWW(\bfs^{(\rho_1)},\bfs^{(\rho_2)})\right),
\end{equation}
using the notation in \eqref{eq6.3}.

Assume that
\textcolor{white}{xxxxxxxxxxxxxxxxxxxxxxxxxxxxxx}
\begin{equation}\label{eq6.16}
2(z-4)\delta-4(A+\delta)\ge\frac{\delta z}{2}.
\end{equation}
Then
\begin{equation}\label{eq6.17}
(z-4)(3r-1)\delta M-4(A+\delta)M
\ge2r(z-4)\delta M-4r(A+\delta)M
\ge\frac{r\delta Mz}{2}.
\end{equation}
Consider a summand in \eqref{eq6.15} with $1\le \rho_1<\rho_2\le q$.
It follows from \eqref{eq6.13} that
\begin{displaymath}
\vert I(\bfs^{(\rho_1)})\cap\XXX\vert-\vert I(\bfs^{(\rho_2)})\cap\XXX\vert\ge\frac{6(\rho_2-\rho_1)\delta M}{z^h},
\end{displaymath}
and so we can use Lemma~\ref{lem6.1} with $r=\rho_2-\rho_1$.
Combining this with \eqref{eq6.17}, we conclude that
\textcolor{white}{xxxxxxxxxxxxxxxxxxxxxxxxxxxxxx}
\begin{equation}\label{eq6.18}
\Upsilon
\le\frac{4c_62^{n_1+z_1i_t}z}{\delta M}\sum_{\rho=1}^q\sum_{r=1}^q\frac{1}{r}
\le\frac{4qc_62^{n_1+z_1i_t}z(1+\log q)}{\delta M}.
\end{equation}
On the other hand, it is easy to see that
\begin{equation}\label{eq6.19}
\sum_{\rho=1}^q\sum_{j\in\JJJ}\lambda\left(T_\theta^jI^*(\bfs^{(\rho)})\right)
=\frac{qc_62^{n_1+z_1i_t}}{2z^h}.
\end{equation}
Combining \eqref{eq6.14}, \eqref{eq6.18} and \eqref{eq6.19}, we obtain estimate
\begin{equation}\label{eq6.20}
q\le\frac{1}{c_62^{n_1+z_1i_t}}\left(\frac{1}{2z^h}-\frac{4z(1+\log q)}{\delta M}\right)^{-1}.
\end{equation}
We shall show later that \eqref{eq6.20} implies $q\le c_9$ for some constant $c_9=c_9(\PPP)>0$ which depends at most on~$\PPP$.

Suppose that neither assertion (i) nor assertion (ii) is valid.
Then for any two adjacent intervals $I(\bfs^*),I(\bfs^{**})\in\frakI(E)$,
\begin{equation}\label{eq6.21}
\vert\vert I(\bfs^*)\cap\XXX\vert-\vert I(\bfs^{**})\cap\XXX\vert\vert<\frac{3\delta M}{z^h}.
\end{equation}
Furthermore, there exist two intervals $I(\bfs'),I(\bfs'')\in\frakI(E)$ such that
\begin{displaymath}
\vert\vert I(\bfs')\cap\XXX\vert-\vert I(\bfs'')\cap\XXX\vert\vert>\frac{c_8\delta M}{z^h}.
\end{displaymath}
Without loss of generality, suppose that $\vert I(\bfs')\cap\XXX\vert>\vert I(\bfs'')\cap\XXX\vert$.
Then the interval $[\vert I(\bfs'')\cap\XXX\vert,\vert I(\bfs')\cap\XXX\vert]$ has length
\begin{displaymath}
\length([\vert I(\bfs'')\cap\XXX\vert,\vert I(\bfs')\cap\XXX\vert])>\frac{c_8\delta M}{z^h},
\end{displaymath}
and contains the subintervals
\begin{equation}\label{eq6.22}
\AAA_\rho=\left(\vert I(\bfs')\cap\XXX\vert-\frac{3\rho\delta M}{z_h},\vert I(\bfs')\cap\XXX\vert-\frac{3(\rho-1)\delta M}{z_h}\right),
\quad
\rho=1,\ldots,\left[\frac{c_8}{3}\right].
\end{equation}
We now go from $I(\bfs')$ to $I(\bfs'')$.
More precisely, consider the collection
\begin{displaymath}
\BBB(\bfs',\bfs'')=\{I(\bfs)\in\frakI(E):\mbox{$I(\bfs)$ lies between $I(\bfs')$ and $I(\bfs'')$}\}
\end{displaymath}
of special intervals of length $z^{-h}$ that lie between $I(\bfs')$ and $I(\bfs'')$.

Note that each interval $\AAA_\rho$ in \eqref{eq6.22} has length $3\delta M/z^h$.
The condition \eqref{eq6.21} now ensures that $\AAA_\rho$ must contain a term of the form $\vert I(\bfs)\cap\XXX\vert$
for some $I(\bfs)\in\BBB(\bfs',\bfs'')$.
Suppose that
\textcolor{white}{xxxxxxxxxxxxxxxxxxxxxxxxxxxxxx}
\begin{displaymath}
\vert I(\bfs^{(\rho)})\cap\XXX\vert\in\AAA_{3\rho},
\quad
\rho=1,\ldots,\left[\frac{c_8}{9}\right].
\end{displaymath}
Then clearly
\begin{displaymath}
\vert I(\bfs^{(\rho)})\cap\XXX\vert-\vert I(\bfs^{(\rho+1)})\cap\XXX\vert\ge\frac{6\delta M}{z^h},
\quad
\rho=1,\ldots,\left[\frac{c_8}{9}\right]-1.
\end{displaymath}
This contradicts the maximality of $q$ if we choose $c_8=9(c_9+2)$.

The proof of the lemma is now complete, subject to verifying the condition \eqref{eq6.16}
and showing that  \eqref{eq6.20} implies $q\le c_9$ for some constant $c_9=c_9(\PPP)>0$.
\end{proof}

In view of Lemma~\ref{lem6.2}, we have two subcases.

\begin{case2a}
There exist an edge $E$ of $\PPP$ and two adjacent intervals $I(\bfs^*),I(\bfs^{**})\in\frakI(E)$ such that the inequality \eqref{eq6.11} holds.
\end{case2a}

\begin{case2b}
For every edge $E$ of $\PPP$ and for any two intervals $I(\bfs'),I(\bfs'')\in\frakI(E)$, the inequality \eqref{eq6.12} holds.
\end{case2b}

Let us first investigate Case~2A.
In view of symmetry, we may assume, without loss of generality, that $I(\bfs^*)$ is the left neighbor of $I(\bfs^{**})$.
Recall that each of
\begin{displaymath}
I(\bfs^*)=\bigcup_{s=1}^zI(\bfs^*,s)
\quad\mbox{and}\quad
I(\bfs^{**})=\bigcup_{s=1}^zI(\bfs^{**},s)
\end{displaymath}
is a union of $z$ subintervals of length $z^{-h-1}$.
Consider the interval
\begin{displaymath}
I^\circ=I(\bfs^*,z)\cup\left(\bigcup_{s=1}^{z-1}I(\bfs^{**},s)\right),
\end{displaymath}
made up of the last subinterval in $I(\bfs^*)$ and every subinterval apart from the last in $I(\bfs^{**})$.
For convenience, write
\begin{displaymath}
I^\circ(1)=I(\bfs^*,z)
\quad\mbox{and}\quad
I^\circ(s)=I(\bfs^{**},s-1),
\quad
s=2,\ldots,z.
\end{displaymath}

\begin{lem}\label{lem6.3}
There exists an integer $s=1,\ldots,z$ such that
\begin{equation}\label{eq6.23}
\left\vert\vert I^\circ(s)\cap\XXX\vert-\frac{\vert I^\circ\cap\XXX\vert}{z}\right\vert>\frac{\delta M}{z^{h+1}}.
\end{equation}
\end{lem}

\begin{proof}
Suppose that
\textcolor{white}{xxxxxxxxxxxxxxxxxxxxxxxxxxxxxx}
\begin{displaymath}
\vert I^\circ(1)\cap\XXX\vert-\frac{\vert I^\circ\cap\XXX\vert}{z}>\frac{\delta M}{z^{h+1}}.
\end{displaymath}
Then clearly \eqref{eq6.23} holds with $s=1$.
Thus we may assume that
\begin{equation}\label{eq6.24}
\vert I^\circ(1)\cap\XXX\vert-\frac{\vert I^\circ\cap\XXX\vert}{z}\le\frac{\delta M}{z^{h+1}}.
\end{equation}
Suppose next that
\textcolor{white}{xxxxxxxxxxxxxxxxxxxxxxxxxxxxxx}
\begin{displaymath}
\vert I^\circ\cap\XXX\vert\le\vert I(\bfs^{**})\cap\XXX\vert+\frac{\delta M}{z^h}.
\end{displaymath}
Combining this with \eqref{eq6.1} and \eqref{eq6.11}, we deduce that
\begin{displaymath}
\vert I^\circ(1)\cap\XXX\vert
>\frac{\vert I(\bfs^*)\cap\XXX\vert}{z}-\frac{\delta M}{z^{h+1}}
\ge\frac{\vert I(\bfs^{**})\cap\XXX\vert}{z}+\frac{2\delta M}{z^{h+1}}
\ge\frac{\vert I^\circ\cap\XXX\vert}{z}+\frac{\delta M}{z^{h+1}},
\end{displaymath}
contradicting \eqref{eq6.24}.
Thus we may assume that
\begin{equation}\label{eq6.25}
\vert I^\circ\cap\XXX\vert>\vert I(\bfs^{**})\cap\XXX\vert+\frac{\delta M}{z^h}.
\end{equation}
It is not difficult to check that
\begin{equation}\label{eq6.26}
\vert I^\circ\cap\XXX\vert=\vert I^\circ(1)\cap\XXX\vert+\vert I(\bfs^{**})\cap\XXX\vert-\vert I(\bfs^{**},z)\cap\XXX\vert.
\end{equation}
Combining \eqref{eq6.25} and \eqref{eq6.26}, and then applying \eqref{eq6.1}, we deduce that
\begin{displaymath}
\frac{\delta M}{z^h}
<\vert I^\circ(1)\cap\XXX\vert-\vert I(\bfs^{**},z)\cap\XXX\vert
<\frac{\vert I^\circ\cap\XXX\vert}{z}-\frac{\vert I(\bfs^{**})\cap\XXX\vert}{z}+\frac{2\delta M}{z^{h+1}}.
\end{displaymath}
so that
\textcolor{white}{xxxxxxxxxxxxxxxxxxxxxxxxxxxxxx}
\begin{equation}\label{eq6.27}
\frac{(z-2)\delta M}{z^{h+1}}<\frac{\vert I^\circ\cap\XXX\vert}{z}-\frac{\vert I(\bfs^{**})\cap\XXX\vert}{z}.
\end{equation}
Finally, using \eqref{eq6.1} again and combining with \eqref{eq6.27}, we conclude that
\begin{displaymath}
\vert I^\circ(2)\cap\XXX\vert
<\frac{\vert I(\bfs^{**})\cap\XXX\vert}{z}+\frac{\delta M}{z^{h+1}}
<\frac{\vert I^\circ\cap\XXX\vert}{z}-\frac{(z-3)\delta M}{z^{h+1}}.
\end{displaymath}
This gives \eqref{eq6.23} with $s=2$, on noting that $z\ge4$.
\end{proof}

Observe that \eqref{eq6.23} is the analog of \eqref{eq5.1} in Case~1.
We can therefore repeat the method of global spreading of the local imbalance via the flow in direction~$\theta$,
and obtain an analog of the estimate \eqref{eq5.25}.

Recall that we are considering Case~2 here.
There exists a set $\TTT_2$ with cardinality $\vert\TTT_2\vert\ge\eta(k-1)/2$
such that for each of the $\vert\TTT_2\vert$ integers $h$ of the form
\eqref{eq4.34}, we have an estimate of the form \eqref{eq5.25}.
Analogs of the estimates \eqref{eq5.26} and \eqref{eq5.27} follow.
Choosing $k$ sufficiently large as in Case~1 now ensures that Case~2A is impossible.
It remains to study Case~2B.

%
%

\section{Completing the proof}\label{sec7}

Before we study Case~2B and deduce Theorem~\ref{thm1}, we need to first analyze all the constants and parameters
that arise from the argument thus far.

First of all, to ensure that Cases 1 and~2A are impossible, we need to choose the integer $k$ to satisfy the inequality
\begin{equation}\label{eq7.1}
k-1>\frac{64c_0^2c_1A^3z}{c_6\eta\delta^3}.
\end{equation}

The constants $c_0,c_2,c_3,c_4,c_5,c_6$ all depend at most on~$\PPP$.
Here $c_0$ is required to satisfy $c_0\ge4$ in the hypotheses of Lemmas \ref{lem2.2} and~\ref{lem2.3},
and must also satisfy the first inequality in \eqref{eq4.30}.
Thus
\begin{displaymath}
c_0\ge\max\left\{4,\frac{16C^\star}{\pi\eps}\right\},
\end{displaymath}
where $C^\star=C^\star(\PPP)$ is the constant of Vorobets in Lemma~\ref{lem2.1}.
On the other hand, $c_2,c_3,c_4$ satisfy \eqref{eq2.4} and \eqref{eq2.8}, whereas $c_5,c_6$ arise respectively from Lemmas \ref{lem2.2}
and~\ref{lem2.3}.
Meanwhile, the constant $c_1$ depends on the direction $\theta$ as well, but has lower and upper bounds $c_2$ and $c_3$ respectively.

The constants $c_8$ and $c_9$ arise from Lemma~\ref{lem6.2} and its proof, with $c_8=9(c_9+2)$.
We next show that the parameter $q$ in \eqref{eq6.20} satisfies $q\le c_9$ for some constant $c_9=c_9(\PPP)>0$, subject to the extra condition
\begin{equation}\label{eq7.2}
\delta M\ge16z^{h+1}(1+\log q).
\end{equation}
Indeed, combining \eqref{eq7.2} with \eqref{eq2.4}, \eqref{eq4.18}, \eqref{eq4.34} and \eqref{eq6.20}, we deduce that
\begin{equation}\label{eq7.3}
q\le\frac{4z^h}{c_62^{n_1+z_1i_t}}=\frac{4c_0^2c_1}{c_6}\le\frac{4c_0^2c_3}{c_6}=c_9,
\end{equation}
a bound that depends at most on~$\PPP$.

As well as ensuring that \eqref{eq7.2} holds, we also need to make sure that the condition \eqref{eq5.10} concerning the parameter $M$ is satisfied,
as are the conditions \eqref{eq5.6}, \eqref{eq5.19} and \eqref{eq6.16} relating the parameters $\delta$, $M$, $M_1$, $z$, $h$ and~$A$.

Given $\eps>0$, we choose $\delta$ to satisfy
\begin{equation}\label{eq7.4}
c_8\delta=9(c_9+2)\delta=\frac{\eps^2}{4}<\frac{\eps}{4},
\end{equation}
so that the exists some constant $c_{10}=c_{10}(\PPP)>0$ such that
\begin{equation}\label{eq7.5}
\delta=c_{10}\eps^2.
\end{equation}
Next, the expression \eqref{eq5.11} for $A$ and the inequality \eqref{eq7.3} for $c_9$ gives the bound
\begin{equation}\label{eq7.6}
A\le c_9.
\end{equation}
Then the condition \eqref{eq6.16}, which is equivalent to $3\delta z\ge8A+24\delta$, is guaranteed if the number~$z$,
which is an integer power of~$2$, is defined by
\begin{equation}\label{eq7.7}
\frac{2c_{11}}{\eps^2}>z\ge\frac{c_{11}}{\eps^2}\ge\frac{3c_9}{\delta}+8
\end{equation}
where the constant $c_{11}=c_{11}(\PPP)>0$ is sufficiently large to satisfy the condition
\begin{equation}\label{eq7.8}
c_{11}\ge\frac{4c_9}{c_{10}}.
\end{equation}
Combining \eqref{eq4.16}, \eqref{eq4.18}, \eqref{eq4.31}, \eqref{eq4.33}, \eqref{eq5.11} and \eqref{eq7.5}--\eqref{eq7.8}, we now obtain
\begin{displaymath}
\frac{8Az^{h+1}}{\delta M_1}
=\frac{4Az2^{n_1+z_1i_t}}{\delta2^{n_1+kz_1}}
=\frac{4A}{\delta z^{k-i_t-1}}
\le\frac{4A}{\delta z}
\le\frac{4c_9}{c_{10}c_{11}}
\le1,
\end{displaymath}
so that \eqref{eq5.19} is satisfied.
Meanwhile, using \eqref{eq2.4}, \eqref{eq4.30} and \eqref{eq7.5}--\eqref{eq7.7}, we have
\begin{displaymath}
\frac{64c_0^2c_1A^3z}{c_6\eta\delta^3}\le\frac{256c_0^2c_3c_9^3c_{11}}{c_6c_{10}^3\eps^9}.
\end{displaymath}
Thus, taking the integer $k$ to satisfy
\begin{equation}\label{eq7.9}
k=\left[\frac{c_{12}}{\eps^9}\right]+2<\frac{2c_{12}}{\eps^9},
\quad\mbox{where}\quad
c_{12}=c_{12}(\PPP)=\max\left\{\frac{256c_0^2c_3c_9^3c_{11}}{c_6c_{10}^3},1\right\},
\end{equation}
ensures that \eqref{eq7.1} holds.
It remains to ensure that \eqref{eq5.6}, \eqref{eq5.10} and \eqref{eq7.2} are satisfied.
For \eqref{eq5.6} and \eqref{eq7.2}, using \eqref{eq2.4}, \eqref{eq4.18}, \eqref{eq4.31}, \eqref{eq4.33}, \eqref{eq4.34},
\eqref{eq7.3} and \eqref{eq7.5}, we see that
\textcolor{white}{xxxxxxxxxxxxxxxxxxxxxxxxxxxxxx}
\begin{align}\label{eq7.10}
&
\max\left\{\frac{64z^{h+1}}{\delta},\frac{16z^{h+1}(1+\log q)}{\delta}\right\}
\le\frac{64c_9z^{h+1}}{\delta}
\nonumber
\\
&\quad
\le\frac{64c_0^2c_3c_9z^{n_1/z_1}z^{i_t+1}}{c_{10}\eps^2}
\le\frac{64c_0^2c_3c_9z^{n_1/z_1}z^k}{c_{10}\eps^2}
=\frac{64c_0^2c_3c_9Nz^k}{c_{10}\eps^2}.
\end{align}
For \eqref{eq5.10}, using \eqref{eq2.4}, \eqref{eq4.16} and \eqref{eq4.18}, we see that
\begin{equation}\label{eq7.11}
\max\{2c_0^2c_1,c_6\}2^{n_1+kz_1}
\le\max\{2c_0^2c_3,c_6\}z^{n_1/z_1}z^k
=\max\{2c_0^2c_3,c_6\}Nz^k.
\end{equation}
Finally, note that in view of \eqref{eq7.7} and \eqref{eq7.9}, we have
\begin{equation}\label{eq7.12}
z^k<\left(\frac{2c_{11}}{\eps^2}\right)^{2c_{12}\eps^{-9}}.
\end{equation}
Now let the constant $C=C(\PPP,\eps)>0$ be an integer satisfying
\begin{equation}\label{eq7.13}
C=C(\PPP;\eps)\ge\max\left\{\frac{64c_0^2c_3c_9}{c_{10}\eps^2},2c_0^2c_3,c_6\right\}\left(\frac{2c_{11}}{\eps^2}\right)^{2c_{12}\eps^{-9}}.
\end{equation}
Choosing $M=CN$, it then follows from \eqref{eq7.10}--\eqref{eq7.13} that \eqref{eq5.6}, \eqref{eq5.10} and \eqref{eq7.2} are satisfied.
This completes the analysis of the constants and parameters in our argument.

\begin{proof}[Completion of the proof of Theorem~\ref{thm1}]
We let $\Gamma(\PPP;N;\eps)$ be the collection $\Omega_N^c(\eta)$ of good directions defined in Section~\ref{sec4}
by \eqref{eq4.16}--\eqref{eq4.23} and \eqref{eq4.30}.
In view of \eqref{eq4.29}, we have $\lambda(\Gamma(\PPP;N;\eps))\ge(1-\eps)2\pi$.
Our estimates thus far apply to all directions $\theta\in\Gamma(\PPP;N;\eps)$.

By choosing the constants and parameters appropriately, we have ensured that Cases 1 and~2A are impossible.
Recall that for Case~2B, in view of the choice \eqref{eq7.4}, for every edge of $\PPP$ and for any two intervals $I(\bfs'),I(\bfs'')\in\frakI(E)$,
the inequality
\begin{displaymath}
\vert\vert I(\bfs')\cap\XXX\vert-\vert I(\bfs'')\cap\XXX\vert\vert\le\frac{\eps M}{4z^h}
\end{displaymath}
holds for any of the special choices of integers $h$ given by \eqref{eq4.34} with cardinality $\vert\TTT_2\vert\ge\eta(k-1)/2$.
Consider a particular choice $h$ and keep it fixed.
Then the inequality \eqref{eq6.1} holds for every integer sequence $\bfs=(s_1,\ldots,s_h)\in\{1,\ldots,z\}^h$ and every integer $s\in\{1,\ldots,z\}$.
Note next from \eqref{eq7.5} and \eqref{eq7.7} that $\delta$ is much smaller than $\eps$ and $z$ is much larger than~$1/\eps$.
This allows us to spread the inequality \eqref{eq6.12} to other intervals of length $z^{-h}$ within $\overline{\psi(E)}$
beyond those special intervals in $\frakI(E)$, with a slightly weaker inequality
\begin{equation}\label{eq7.14}
\vert\vert I'\cap\XXX\vert-\vert I''\cap\XXX\vert\vert\le\frac{\eps M}{3z^h},
\end{equation}
where $I'$ and $I''$ are intervals in $\overline{\psi(E)}$ with length~$z^{-h}$, obtained from the special intervals in $\frakI(E)$ by shifts
by integer multiples of $z^{-h-1}$.

The estimate \eqref{eq7.14} has the message that every edge $E$ of $\PPP$
exhibits \textit{almost uniform distribution of the hitting points with its own density}.
However, these edge-dependent densities cannot be substantially different, as the geodesic flow goes from one edge to the next,
transporting the density from one to the next, and thus enforces almost equality.
Thus the finite number of edge-dependent densities turn out to be almost the same, with $\eps$ error.
This completes the proof.
\end{proof}

%
%

\end{document}